\documentclass[12pt,fleqn]{article}

\usepackage{todonotes}
{}

\usepackage[english]{babel}

\usepackage[utf8]{inputenc}
\usepackage[T1]{fontenc}

\usepackage{mathtools}
\usepackage{amsfonts,amsmath, amsthm,amssymb,esint}
\usepackage{dsfont}

\allowdisplaybreaks

\usepackage{Alegreya}
\usepackage{float}
\usepackage{authblk}
\usepackage{tikz}
\usepackage{color}
\definecolor{DB}{rgb}{0.0,0.0,0.8} 
\definecolor{DG}{rgb}{0.0,0.55,0.14}
\definecolor{DR}{rgb}{0.5,0,0.07}

\usepackage[hyperindex=true,colorlinks=true, urlcolor={DB}, linkcolor={DB}, menucolor={DB}, citecolor={DG}, anchorcolor={DG},linktoc=all]{hyperref}

\usepackage{url}

\usepackage{amsrefs}

\usepackage{bbding}
\newcommand{\verti}[1]{{\left\vert #1 \right\vert}}
\newcommand{\vertii}[1]{{\left\vert\kern-0.25ex\left\vert #1 \right\vert\kern-0.25ex\right\vert}}  
\newcommand{\vertiii}[1]{{\left\vert\kern-0.25ex\left\vert\kern-0.25ex\left\vert #1 \right\vert\kern-0.25ex\right\vert\kern-0.25ex\right\vert}}

\def\bt{\begin{theo}}
\def\et{\end{theo}}
\def\bpr{\begin{prop}}
\def\epr{\end{prop}}
\def\bl{\begin{lemma}}
\def\el{\end{lemma}}
\def\bc{\begin{coro}}
\def\ec{\end{coro}}
\def\br{\begin{rema}}
\def\er{\end{rema}}
\def\bp{\begin{proof}}
\def\ep{\end{proof}}

\numberwithin{equation}{section}

\usepackage{mathrsfs}
\usepackage{stmaryrd}

\usepackage{graphicx}

\usepackage[margin=25mm]{geometry}

\def\R{{\mathbb R}}
\def\N{{\mathbb N}}

\def\Z{{\mathbb Z}}

\def\fo{\forall\,}

\def\bes{\begin{equation*}}
\def\ees{\end{equation*}}
\def\be{\begin{equation}}
\def\ee{\end{equation}}
\def\ba{\begin{aligned}}
\def\ea{\end{aligned}}

\def\supp{\operatorname{supp}}

\usepackage{enumerate}
\def\ben{\begin{enumerate}}
\def\een{\end{enumerate}}

\theoremstyle{definition}
\newtheorem{prop}{Proposition}[section]
\newtheorem{theo}[prop]{Theorem}
\newtheorem{coro}[prop]{Corollary}
\newtheorem{lemma}[prop]{Lemma}
\theoremstyle{definition}

\theoremstyle{definition}
\newtheorem{rema}[prop]{Remark}
\theoremstyle{definition}
\newtheorem{defi}[prop]{Definition}

\def\ep{\end{proof}}
\def\bp{\begin{proof}}

\newcommand\blfootnote[1]{%
  \begingroup
  \renewcommand\thefootnote{}\footnote{#1}%
  \addtocounter{footnote}{-1}%
  \endgroup
}

\usepackage[autostyle=true]{csquotes}
\title{Existence of extremal functions in higher-order affine Sobolev inequalities}
\author[Tristan]{Tristan Bullion-Gauthier*}
\date{}

\begin{document}
\maketitle
\begin{abstract}
    In this article, we prove the existence of extremal functions in higher-order affine Sobolev inequalities. Proofs rely on concentration-compactness methods in spaces of integer or fractional regularity. The tools we use, available in spaces of arbitrary regularity, might be of independent interest.
\end{abstract}
\blfootnote{*Universite Claude Bernard Lyon 1, CNRS, Centrale Lyon, INSA Lyon, Université Jean Monnet, ICJ UMR5208, 69622 Villeurbanne, France.  
\url{bullion@math.univ-lyon1.fr}}
\blfootnote{Acknowledgments. I thank \'Oscar Dom\'inguez and Petru Mironescu for their mathematical advice and suggestions to improve the presentation.}

\section{Introduction}
This is a follow-up of \cite{bullion2026higher}, where we have addressed the question of the validity of higher-order affine Sobolev inequalities. These inequalities involve the \enquote{affine} energies defined by 
\bes
 \mathscr{E}_{s,p}(f):=\left(\int_{\mathbb{S}^{N-1}} \left(\int_{0}^{\infty}t^{-sp-1} \vertii{\Delta^{\lfloor s \rfloor +1}_{t\xi}f}_{L^p(\R^N)}^p \, dt\right)^{-N/sp}\,  d\mathscr{H}^{N-1}(\xi) \right)^{-s/N},
\ees
 if $s>0$ is non-integer, respectively
\bes
    \mathscr{E}_{s,p}(f):= \left(\int_{\mathbb{S}^{N-1}} \left(\int_{\R^N}\verti{\partial^s_{\xi} f(x)}^p \ dx\right)^{-N/sp} d \mathscr{H}^{N-1}(\xi)\right)^{-s/N}, 
    \ees
if $s\ge 1$ is an integer. 

One of the contributions in \cite{bullion2026higher} asserts that, for each $s>0$ and $1\leq p<\infty$ satisfying $sp<N$, with $p>1$ if $s\geq 2$ is an integer, there exists $S>0$ such that 
\be \label{aff sob}
S \vertii{f}_{L^{q}(\R^N)} \leq \mathscr{E}_{s,p}(f)\text{,} \ \fo f \in L^{q}(\R^N) \cap \dot{W}^{s,p}(\R^N),
\ee
 where $q\coloneq Np/(N-sp)$. Note that \eqref{aff sob} is a strengthened version of the classical Sobolev inequality that is invariant under affine transformations. In the following, we implicitly assume that $S$ is the best constant in \eqref{aff sob}. In some sense, \eqref{aff sob} is a generalization of the remarkable results of Lutwak, Yang, and Zhang \cites{zhang1999affine,lutwak2002sharp}, and of their fractional analogues obtained by Haddad and Ludwig  \cites{haddad2024affine,haddad2025affine} corresponding, respectively, to  $s=1$ and to $0<s<1$. A major difference between \cite{bullion2026higher} and these works is that, in \cite{bullion2026higher}, the value of $S$ is not explicit. 

  In the case where $0<s\leq 1$, rearrangement or convex geometry techniques were used in \cites{haddad2024affine,haddad2025affine,lutwak2002sharp} to characterize extremal functions in \eqref{aff sob} (for which there is equality in \eqref{aff sob}). Such tools are unavailable when $s>1$ and extremal functions may \textit{a priori} not exist. The goal of this article is to prove the existence of extremal functions in \eqref{aff sob} in the general setting $s>0$.
  
More specifically, we prove the following. 
\begin{theo}\label{exist sob}
    Let $s>0$ and $1\leq p<\infty$ be such that $sp<N$. When $s$ is an integer, we assume that $p>1$. There exists $f \in \dot{W}^{s,p}(\R^N)$ such that  $\vertii{f}_{L^{q}(\R^N)}=1$ and 
    $\mathscr{E}_{s,p}(f)=S$.
    
    Moreover, when $p>1$, minimizing sequences for $ \inf\{\mathscr{E}_{s,p}(g); \ \vertii{g}_{L^{q}(\R^N)}=1 \}$ are relatively compact in $\dot{W}^{s,p}$ up to unimodular transformations, translations and dilations.

\end{theo}

Our proof of Theorem \ref{exist sob} in the integer case is based on the concentration-compactness method due to Lions \cites{lions1984concentration,lions1985concentration}. In the fractional case, we rely on a variant of this method presented by Bellazzini, Frank, and Visciglia in   \cite{bellazzini2014maximizers}, and used by Zhang in   \cite{zhang2021optimizers}. It is based on refined Sobolev inequalities (with roots in the works of Gérard \cite{gerard1998description}) and adequate versions of the Brezis-Lieb lemma, the latter depending on the variational problem under consideration (in our case, Lemma \ref{uniform BL} will be instrumental). In both cases, we need refined and affine versions of certain classical results. These results allow us to deal with the \enquote{anisotropic} features of affine energies (see, e.g., Proposition \ref{obser} for an observation on the link between affine and \enquote{anisotropic} Sobolev energies). Some results obtained in this direction might be of independent interest  (see, e.g., Lemmas \ref{uniform BL}  and \ref{affin CC}).

Our paper is organized as follows. In Section \ref{affine sec}, we collect some auxiliary results used in the proof of Theorem \ref{exist sob}. Sections \ref{frac} and \ref{int} are devoted to the proof of Theorem \ref{exist sob} in the fractional and integer cases, respectively. 
\section{Basic facts about affine energies and weak convergence}\label{affine sec}
Otherwise is stated, all function spaces are defined on $\R^N$.
Given $s>0$ non-integer and $1 \leq p<\infty$, we denote by $\dot{W}^{s,p}$ the space of functions $f \in L^1_{loc}$ such that 
\bes
\verti{f}_{W^{s,p}}^p\coloneq \int_{\R^N}\frac{\vertii{\Delta^{\lfloor s \rfloor +1}_h f}_{L^p}^p}{\verti{h}^{sp+N}} \ dh<\infty.
\ees
Given $s$ an integer and $1 \leq p<\infty$, we denote by $\dot{W}^{s,p}$ the space of functions $f \in W^{s,1}_{loc}$ such that 
\bes
\verti{f}_{W^{s,p}}^p\coloneq \int_{\R^N} \vertii{D^s_x f}^p \ dx<\infty,
\ees 
where $\vertii{\eta}\coloneq \sup_{\verti{x_1}\leq 1,\dots, \verti{x_s}\leq 1} \verti{\eta(x_1,\dots,x_s)}$, for each $s$-linear form $\eta$. Given $f\in W^{s,1}_{loc}$ and $\xi \in \mathbb{S}^{N-1}$, we will also denote by $\partial^s_{\xi} f$ the function defined for a.e. $x\in \R^N$ by
\bes
\partial^s_{\xi}f(x) = D_x^sf(\xi,\dots,\xi)= \sum_{\verti{\alpha}=s} \partial^{\alpha} f (x). 
\ees
We now present the main features of affine energies. 
\begin{prop}(\cite[Lemma 3.1, Proposition 3.2]{bullion2026higher}) Let $0<s<\infty$ and $1 \leq p<\infty$.
    \begin{enumerate}[(1)]
        \item For each $f\in \dot{W}^{s,p}$ and $T \in \text{SL}_N$, we have $\mathscr{E}_{s,p}(f\circ T)=\mathscr{E}_{s,p}(f)$.
        \item There exists $\alpha_{s,p,N}<\infty$ such that
        \bes
        \mathscr{E}_{s,p}(f) \leq \alpha_{s,p,N}\verti{f}_{W^{s,p}}, \ \fo f\in \dot{W}^{s,p}.
        \ees
    \end{enumerate}
\end{prop}
In relation with item (2) of the previous proposition, we point out that, in general,  $\mathscr{E}_{s,p}$ and $\verti{\cdot}_{W^{s,p}}$ are not equivalent (see \cite[Theorem 2]{haddad2021affine}). However, the affine energy of a function is always equivalent to the Sobolev semi-norm of a suitable transformation of this function, thanks to the following lemma.
\begin{lemma}\label{equili}(\cite[Theorem 1.4]{bullion2026higher})
    Let $s>0$ and $1\leq p<\infty$.  For each $f \in \dot{W}^{s,p}$, there exists $T_f \in \text{SL}_N$ such that 
    \bes
    C\verti{f\circ T_f}_{W^{s,p}}\leq \left(\int_{0}^{\infty} t^{-sp-1} \vertii{\Delta_{t\xi}^{\lfloor s \rfloor+1} (f\circ T_f)}_{L^p}^p \ dt \right)^{1/p}, \ \fo \xi \in \mathbb{S}^{N-1}, 
    \ees
    if $s$ is non-integer, respectively
    \bes
     C\verti{f\circ T_f}_{W^{s,p}}\leq \vertii{\partial^s_{\xi} f}_{L^p}, \ \fo \xi \in \mathbb{S}^{N-1}, 
    \ees
    if $s$ is an integer, with $p>1$ if $s>1$.

    In the above inequalities, $C>0$ depends on $s,p,N$ but not on $f$.
\end{lemma}
We also mention the following observation, which illustrates the link between affine and \enquote{anisotropic} fractional Sobolev semi-norms. It is presented in \cite{haddad2024affine}, when $0<s<1$. Similar results hold when $s=1$, see \cite{lutwak2006optimal}.
\begin{prop}\label{obser}
    Let $s$ be non-integer, $1 \leq p<\infty,$ and $f\in \dot{W}^{s,p}$. 
    For each measurable $\rho$ such that
    $\rho \geq 0$,  $\rho \ \text{is} \ (-1)\text{ - homogeneous}$, and \bes \displaystyle \int_{\mathbb{S}^{N-1}} \rho^N(\xi) \ d \mathscr{H}^{N-1}(\xi) =1,\ees
    we have
    \be\label{result}
    \int_{\R^N}\vertii{\Delta_h^{\lfloor s \rfloor +1} f}_{L^p}^p \rho(h)^{N+sp} \ dh \geq \mathscr{E}_{s,p}(f)^p.
    \ee
    If $\verti{f}_{W^{s,p}}\neq 0$, there is equality in \eqref{result} if and only if 
    \bes
    \rho(\xi) = \mathscr{E}_{s,p}(f)^{1/s} \left(\int_{0}^{\infty} t^{-sp-1} \vertii{\Delta^{\lfloor s \rfloor +1}_{t\xi} f}_{L^p}^p \ dt\right)^{-1/sp}, \ \text{for a.e.} \  \xi \in \mathbb{S}^{N-1}. 
    \ees
    \end{prop}
\begin{proof}
    Using radial coordinates and the homogeneity of $\rho$, we find
    \bes
     \int_{\R^N}\vertii{\Delta_h^{\lfloor s \rfloor +1} f}_{L^p}^p \rho(h)^{N+sp} \ dh = \int_{\mathbb{S}^{N-1}} \left(\int_{0}^{\infty}t^{-sp-1} \vertii{\Delta_{t\xi}^{\lfloor s \rfloor +1} f}_{L^p}^p \ dt \right) \rho(\xi)^{N+sp} \ d\mathscr{H}^{N-1}(\xi). 
    \ees
   Applying the reverse H\"older inequality 
    \be\label{rev Hold}
    \ba
    \int_{\mathbb{S}^{N-1}} & u v\ d\mathscr{H}^{N-1}
    \\
    &\geq  \left( \int_{\mathbb{S}^{N-1}} u^{-N/sp}\ d\mathscr{H}^{N-1}\right)^{-sp/N} \left(\int_{\mathbb{S}^{N-1}} v^{N/(N+sp)} \ d\mathscr{H}^{N-1}\right)^{(N+sp)/N}, 
    \\ 
     &\fo u \geq 0, \ v\geq 0 \ \text{measurable}, 
    \ea
    \ee
    to
    \bes
    u(\xi)\coloneq  \int_{0}^{\infty}t^{-sp-1} \vertii{\Delta_{t\xi}^{\lfloor s \rfloor +1} f}_{L^p}^p \ dt, \ v(\xi)\coloneq \rho(\xi)^{N+sp}, \ \fo \xi \in \mathbb{S}^{N-1},
    \ees
    we therefore find
    \bes
    \int_{\R^N}\vertii{\Delta_h^{\lfloor s \rfloor +1} f}_{L^p}^p \rho(h)^{N+sp} \ dh  \geq \mathscr{E}_{s,p}(f)^p \left(\int_{\mathbb{S}^{N-1}} \rho(\xi)^N \ d\mathscr{H}^{N-1}(\xi) \right)^{(N+sp)/N}= \mathscr{E}_{s,p}(f)^p.
    \ees
    Hence, there is equality in \eqref{result} if and only if there is equality in the application of the reverse H\"older inequality \eqref{rev Hold}. In turn, when $\verti{f}_{W^{s,p}}\neq 0$, this is equivalent to
    \bes
    \ba
    \rho(\xi) &=\lambda \left( \int_{0}^{\infty}t^{-sp-1} \vertii{\Delta_{t\xi}^{\lfloor s \rfloor +1} f}_{L^p}^p \ dt \right)^{-1/sp}, \ \text{for almost every} \ \xi \in  \mathbb{S}^{N-1},
    \ea
    \ees
    for some $\lambda>0$. In this case, the condition 
    \bes
    \int_{\mathbb{S}^{N-1}} \rho^N(\xi) \ d\mathscr{H}^{N-1}(\xi) =1 
    \ees
    implies that $\lambda=\mathscr{E}_{s,p}(f)^{1/s}$.
\end{proof}  
A direct consequence of the previous proposition and of \eqref{aff sob} is the following.
\begin{prop}\label{ani prop}
    Let $s$ be non-integer and $1 \leq p<\infty$ be such that $sp<N$.  For each measurable $\rho$ such that
    $\rho \geq 0$,  $\rho \ \text{is} \ (-1)\text{ - homogeneous}$, and \bes \displaystyle \int_{\mathbb{S}^{N-1}} \rho^N(\xi) \ d \mathscr{H}^{N-1}(\xi) =1,\ees
    we have
    \be\label{ani}
        S \vertii{f}_{L^q} \leq \left(\int_{\R^N}\vertii{\Delta_h^{\lfloor s \rfloor +1} f}_{L^p}^p \rho(h)^{N+sp} \ dh \right)^{1/p}, \ \fo f \in L^q \cap \dot{W}^{s,p}.
        \ee
 $S$  (the best constant in \eqref{aff sob}) is the best constant such that \eqref{ani} holds independently of $\rho$.
\end{prop}

In what follows, we will denote the weak-$*$ convergence of finite Borel measures by $\overset{*}{\rightharpoonup}$. 

We use two notions of weak convergence, adapted to integer, respectively non-integer order spaces. If $s$ is an integer, we use the classical notion of weak convergence and say that $f_n \rightharpoonup f$ in $\dot{W}^{s,p}$ if $\partial^s_{\xi}f_n \rightharpoonup \partial^s_{\xi} f$ in $L^p$, for each $\xi \in \mathbb{S}^{N-1}$. 
If $s$ is non-integer, it will be more convenient to work in the space 
\bes
\mathring{W}^{s,p}\coloneq L^{q} \cap \dot{W}^{s,p}.
\ees Given $s>0$ and $1\leq p<\infty$ such that $sp<N$, we have the Sobolev embedding
\be\label{sob}
C \vertii{f}_{L^q} \leq \verti{f}_{W^{s,p}}, \ \fo f \in \mathring{W}^{s,p},
\ee
see, e.g., \cite[Theorems 12.9, 17.49]{leoni2017first}. 
Hence, $(\mathring{W}^{s,p}, \verti{\cdot}_{W^{s,p}})$ is a Banach space (while $\dot{W}^{s,p}$ is only a semi-normed space). 

We use the following 
\begin{defi}\label{w conv frac}
 Let $s$ be non-integer and $1 \leq p<\infty$ satisfy $sp<N$. We say that a sequence $(f_n) \subset \mathring{W}^{s,p}$ weakly converges to $f$ in $\mathring{W}^{s,p}$  if $\sup_n \verti{f_n}_{W^{s,p}}<\infty$ and $f_n \to f$ in $L^p_{\text{loc}}$. In that case, we write $f_n \rightharpoonup f$ in $\mathring{W}^{s,p}$.
\end{defi}

Definition \ref{w conv frac} is convenient as the following results show.
 \begin{lemma}\label{weak comp}
     Let $s$ be non-integer and $1 \leq p<\infty$ satisfy $sp<N$. If $(f_n) \subset \mathring{W}^{s,p}$ is such that $\sup_n \verti{f_n}_{W^{s,p}}<\infty$, then  there exists a subsequence $(f_{n_k})$ of $(f_n)$, and $f \in \mathring{W}^{s,p}$, such that $f_{n_k} \rightharpoonup f$ in $\mathring{W}^{s,p}$.
 \end{lemma}
 \begin{proof}
   By the Sobolev embedding \eqref{sob} and H\"older's inequality, we have, for each $R>0$
   \bes
    \vertii{f_n}_{L^p(B(0,R))}\leq C\vertii{f_n}_{L^q(B(0,R))} \leq C \verti{f_n}_{W^{s,p}} \leq C \sup_n \verti{f_n}_{W^{s,p}} < \infty, \ \fo n \in \N.
   \ees
    We have, for each $R>0$, $\sup_n \vertii{f_n}_{L^p(B(0,R))} + \verti{f_n}_{W^{s,p}}<\infty$. Therefore, the Rellich-Kondrachov theorem and a diagonal extraction procedure yield a subsequence $(f_{n_k})$ and $f \in L^p_{\text{loc}}$ such that $f_{n_k} \to f$ in $L^p_{\text{loc}}$. 

    By Fatou's lemma and the Sobolev embedding, we have \bes \int_{\R^N} \verti{f(x)}^{q} \ dx \leq \liminf_k \int_{\R^N} \verti{f_{n_k}(x)}^{q} \ dx \leq C\sup_n\verti{f_n}_{W^{s,p}}^q <\infty, \ees
    and $f \in L^q$. This argument also yields $f \in \dot{W}^{s,p}$, and thus $f\in \mathring{W}^{s,p}$.
 \end{proof}
 Another straightforward consequence of the Sobolev embedding \eqref{sob} is
 \begin{lemma}
     Let $s$ be non-integer and $1\leq p <\infty$ satisfy $sp<N$. If $(f_n) \subset \mathring{W}^{s,p}$ weakly converges to $f  \in \mathring{W}^{s,p}$, then $f_n \rightharpoonup f$ in $L^q$.
 \end{lemma}
Before going further, we recall the Littlewood-Paley characterization of homogeneous Besov spaces (for more details see, e.g., \cite[Chapter 5]{triebel2010theory}, \cite[Chapter 2]{bahouri2011fourier}).

We denote by $\mathscr{S}$ the space of Schwartz functions and let 
\bes
\mathscr{Z} \coloneq \{ \Psi \in \mathscr{S}; \ \partial^{\alpha}\Psi(0)=0, \ \fo \alpha \in \N^N \}.
\ees
$\mathscr{Z}$ is equipped with the Fréchet topology inherited from $\mathscr{S}$. We may consider its topological dual $\mathscr{Z}'$.
Let $(\varphi_j)_{j \in \Z}$ be a sequence of functions such that: 
\begin{itemize}
    \item $\supp \varphi_j \subset B(0,2^{j+1}) \setminus B(0,2^{j-1})$ for each $j \in \Z$.
    \item For each multi-index $\alpha$, there exists $C_{\alpha}<\infty$ such that $\verti{\partial^{\alpha}\varphi_j(x)}\leq C_{\alpha}2^{-j\verti{\alpha}}$, for each $x\in \R^N$, $j \in \Z $.
    \item For each $x \in \R^N \setminus \{0\}$, $\sum_{j\in \Z} \varphi_j(x) = 1.$
\end{itemize}
Given $s\in \R$, $1 \leq p<\infty ,1 \leq q <\infty$, we denote by $\dot{B}^{s}_{p,q}$ the space of distributions $f\in \mathscr{Z}'$ such that 
\bes
\ba
\verti{f}_{B^{s}_{p,q}}^q &\coloneq \sum_{j \in \Z} 2^{-sjq}\vertii{\mathscr{F}^{-1}(\varphi_j \mathscr{F}f)}_{L^p}^q  < \infty,
\ea
\ees
where $\mathscr{F}$ is the Fourier transform and $\mathscr{F}^{-1}$ is the inverse Fourier transform. In relation with this definition of Besov spaces, we have the following result, see e.g. \cite[Theorem 2, Section 5.2.3]{triebel2010theory}. 
\begin{lemma}\label{paley}
Let $s$ be non-integer and $1 \leq p<\infty$ be such that $sp<N$.
There exist $0<C_1 \leq  C_2<\infty$ such that
\bes
C_1\verti{f}_{B^{s}_{p,p}}\leq \verti{f}_{W^{s,p}} \leq C_2\verti{f}_{B^{s}_{p,p}}, \ \fo f \in \mathring{W}^{s,p}.
\ees
\end{lemma}
 The $\dot{B}^{s}_{p,p}$ norm is particularly useful to prove density results in $\mathring{W}^{s,p}$. Combining Lemma \ref{paley} and \cite[Proposition 2.27]{bahouri2011fourier}, we have
 \begin{lemma}\label{density}
 Let $s$ be non-integer and $1 \leq p<\infty$ be such that $sp<N$. For each $f \in \mathring{W}^{s,p}$, there exists $(f_n) \subset \mathscr{S}$ such that $\verti{f_n - f}_{W^{s,p}} \to 0$, as $n \to \infty$.
 \end{lemma}

We now gather some lemmas and start by recalling the reverse Minkowski inequality which will be instrumental in this article, see \cite[Theorem 198]{hardy1952inequalities}.
Let $(X,\mu)$ be a measure space and $\alpha<0$. We have
\be \label{reverse minko}
\vertii{f}_{L^{\alpha}}+\vertii{g}_{L^{\alpha}} \leq \vertii{f+g}_{L^{\alpha}},
\ee
for each $\mu-$measurable $f,g$ such that $f\geq 0$, $g\geq 0$.

For the convenience of the reader, we also state a straightforward consequence of Fatou's lemma that we will use repeatedly.
\begin{lemma}\label{rev fatou}
    Let $(X,\mu)$ be a measure space and $(f_n)$ be a sequence of positive functions. If there exists $g \in L^1(X,\mu)$ such that $f_n \leq g$, for each $n \in \N$, then 
    \bes
    \limsup_n \int_{X} f_n \ d\mu \leq \int_{X} \limsup_n f_n \ d \mu.
    \ees
\end{lemma}

We will frequently use the following lemmas when studying affine energies.
\begin{lemma}\label{dir domination} (\cite[Theorem, Section 2.5.13]{triebel2010theory})
    Let $s>0$ be non-integer and $1 \leq p<\infty$. There exists $C>0$ such that, for each $f \in \dot{W}^{s,p}$,
    \bes
    C \left(\int_{0}^{\infty}t^{-sp-1} \vertii{\Delta_{t\xi}^{\lfloor s \rfloor +1} f}_{L^p}^p dt\right)^{1/p} \leq \verti{f}_{W^{s,p}}, \ \fo \xi \in \mathbb{S}^{N-1}.
    \ees
\end{lemma}
\begin{lemma}\label{no constant direction} (\cite[Lemma 2.12] {bullion2026higher})
    Let $1 \leq p < \infty$.
    \begin{enumerate}[(a)]
    \item Let $s$ be non-integer.  If $f \in \dot{W}^{s,p}$  is such that
    \begin{flalign*}
        \inf_{\xi \in \mathbb{S}^{N-1}} \int_{0}^{\infty}t^{-sp-1} \vertii{\Delta_{t\xi}^{\lfloor s \rfloor +1} f}_{L^p}^p dt =0,
    \end{flalign*}
    then $\verti{f}_{W^{s,p}}=0$.
    \item Let $ s $ be an integer.
    If $f \in \dot{W}^{s,p}$ is such that 
    \begin{flalign*} 
    \inf_{\xi \in \mathbb{S}^{N-1}} \int_{\R^N} \verti{\partial_{\xi}^s f(x)}^p \ dx = 0,
    \end{flalign*}
    then $\verti{f}_{W^{s,p}}=0$.
\end{enumerate}
\end{lemma}

\begin{lemma}\label{dir weak conv}
   Let $s$ be non-integer and $1\leq p<\infty$ satisfy $sp<N$.  Let $(f_n) \subset \mathring{W}^{s,p}$, $f \in \mathring{W}^{s,p}$, and consider the maps
    \bes
    \ba
    &F_{n,\xi}: (t,x) \mapsto t^{-s-1/p}\Delta^{\lfloor s \rfloor +1}_{t\xi}f_n(x) , \ \xi \in \mathbb{S}^{N-1}, n \in \N, \\ 
    & F_{\xi}: (t,x) \mapsto t^{-s-1/p}\Delta^{\lfloor s \rfloor +1}_{t\xi}f(x) , \ \xi \in \mathbb{S}^{N-1}.
    \ea
    \ees
    If $f_n \rightharpoonup f$ in $\mathring{W}^{s,p}$ and $\xi_n \to \xi \in \mathbb{S}^{N-1}$, then $F_{n, \xi_n} \rightharpoonup F_{\xi}$ in $L^p(\R^{+} \times \R^N)$ when $p>1$,  $F_{n, \xi_n} \overset{*}{\rightharpoonup}  F_{\xi}$ when $p=1$.  
\end{lemma}
\begin{proof}

 First note that, for each $g \in \dot{W}^{s,p}$, $\varphi \in L^{p'}(\R^+ \times \R^N)$ ($1/p+1/p'=1)$, and $\omega \in \mathbb{S}^{N-1},$ we have
\bes\label{apriori}
\ba
\int_{0}^{\infty}\int_{\R^N}t^{-s-1/p}\Delta^{\lfloor s \rfloor +1}_{t\omega} g(x) \varphi(t,x) \ dxdt &\leq \left(\int_{0}^{\infty}t^{-sp-1}\vertii{\Delta^{\lfloor s \rfloor +1}_{t\omega} g}_{L^p} \ dt \right)^{1/p}\vertii{\varphi}_{L^{p'}(\R^+ \times \R^N)} \\ & \leq C \verti{g}_{W^{s,p}} \vertii{\varphi}_{L^{p'}(\R^+ \times \R^N)},
\ea
\ees
where we used H\"older's inequality for the first inequality, and Lemma \ref{dir domination} for the second one.
Therefore, by density of $C_c^{\infty}$ in $L^{p'}$ when $p>1$, in $C_c$ when $p=1$, it suffices to show that 
\be \label{weak smooth conv}
\int_{0}^{\infty}\int_{\R^N}t^{-s-1/p}\Delta^{\lfloor s \rfloor +1}_{t\xi_n} f_n(x) \varphi(t,x) \ dxdt \to  \int_{0}^{\infty}\int_{\R^N}t^{-s-1/p}\Delta^{\lfloor s \rfloor +1}_{t\xi} f(x) \varphi(t,x) \ dxdt,
\ee
as $n \to \infty$,
for each $\varphi \in C_c^{\infty}([0,\infty) \times \R^N)$ (since the maps $F_{n,\xi_n}$ are bounded in $L^p(\R^{+}\times \R^N$), as a consequence of the boundedness of $(f_n)$ in $\dot{W}^{s,p}$, by Lemma \ref{dir domination}.

Let $\varphi \in C_c^{\infty}([0,\infty) \times \R^N)$. It is supported in $[0,a)\times B(0,R)$, for some $a,R>0$. Since $f_n \rightharpoonup f$ in $\mathring{W}^{s,p}$, we have $f_n \to f$ in $L^p(B(0,R))$. Therefore, there exists  $g\in L^1(B(0,R))$ such that, up to a subsequence, 
\be\label{dom}
\verti{f_n}^p(x)  \leq g(x),
\ee
for a.e. $x\in B(0,R)$.
By a standard argument, it suffices to prove \eqref{weak smooth conv} along a sequence satisfying \eqref{dom}.

For each $n \in \N$,  we notice that
\bes 
\ba
\int_{0}^{\infty}&\int_{\R^N}t^{-s-1/p}\Delta^{\lfloor s \rfloor +1}_{t\xi_n} f_n(x) \varphi(t,x) \ dxdt = \int_{0}^{\infty}\int_{\R^N}t^{-s-1/p} f_n(x) \Delta^{\lfloor s \rfloor +1}_{-t\xi_n}\varphi(t,x) \ dxdt.
\ea
\ees
 We have
\bes 
t^{-s-1/p} f_n(x) \Delta^{\lfloor s \rfloor +1}_{-t\xi_n}\varphi(t,x) \to t^{-s-1/p} f(x) \Delta^{\lfloor s \rfloor +1}_{-t\xi}\varphi(t,x), \ \text{for a.e. } \ (t,x)  \in \R^{+}\times \R^N.
\ees
  We also have
\bes \ba \label{1st dom}
t^{-s-1/p} \verti{f_n(x) \Delta^{\lfloor s \rfloor +1}_{-t\xi}\varphi(t,x)} &\leq  \vertii{D^{\lfloor s \rfloor +1} \varphi }_{L^{\infty}}\verti{f_n(x)} t^{\lfloor s \rfloor -s+1-1/p} \mathds{1}_{[0,a) \times B(0,R+\left(\lfloor s \rfloor +1) a\right)}(t,x)
\\ & \leq C \verti{g(x)}^{1/p} t^{\lfloor s \rfloor -s+1-1/p}  \mathds{1}_{[0,a) \times B(0,R+\left(\lfloor s \rfloor +1) a\right)}(t,x)
\ea
\ees
for a.e. $(t,x) \in \R^{+} \times \R^N$, for each $n \in \N $.
The right-hand side in the last inequality is $L^1(\R^{+} \times \R^N)$. Therefore, dominated convergence yields \eqref{weak smooth conv}. 
\end{proof}

In the integer case, we have the following analogue of Lemma \ref{dir weak conv}.
\begin{lemma} \label{weak int conv}
    Let $s$ be an integer and $1 \leq p < \infty$.  If $(f_n) \subset \dot{W}^{s,p}$ is such that $f_n \rightharpoonup f$ in $\dot{W}^{s,p}$, and $(\xi_n) \subset \mathbb{S}^{N-1}$ is such that $ \xi_n \to \xi$, then $\partial^s_{\xi_n} f_n \rightharpoonup \partial^s_{\xi} f$ in $L^p$.
\end{lemma}
\begin{proof}
    This is a straightforward consequence of the inequality
    \be \label{Lip}
     \vertii{\partial_{\xi}^s f_n - \partial^s_{\omega} f_n}_{L^p} \leq C \verti{\xi - \omega}, \fo \xi, \omega \in \mathbb{S}^{N-1},
    \ee
    where $C$ is independent of $n$,$\xi$, and $\omega$. 

    We obtain \eqref{Lip} by observing that
    \bes
    \ba
    \vertii{\partial_{\xi}^s f_n - \partial^s_{\omega} f_n}_{L^p}&= \vertii{D^sf_n (\xi, \dots,\xi) - D^sf_n (\omega, \dots,\omega)}_{L^p} \\ & \leq \sum_{k=0}^{s-1} \vertii{D^s f_n([\omega]^k, [\xi]^{s-k}) - D^sf_n([\omega]^{k+1}, [\xi]^{s-(k+1)})}_{L^p}  
    \ea
    \ees
 where, for each $0 \leq k\leq s-1$, 
 \bes
 ([\omega]^k, [\xi]^{s-k})\coloneq (\underbrace{\omega,\dots,\omega}_{k \ \text{ times}},\underbrace{\xi,\dots,\xi}_{(s-k) \ \text{ times}}).
 \ees
  We thus have
 \bes
  \vertii{\partial_{\xi}^s f_n - \partial^s_{\omega} f_n}_{L^p} \leq s \verti{\omega-\xi} \vertii{D^s f_n}_{L^p}. \qedhere
 \ees
\end{proof}
\begin{lemma}\label{far from 0}
    Let $s>0$, $1 \leq p < \infty$, $(f_n) \subset \dot{W}^{s,p}$, and $f \in \dot{W}^{s,p}$ be such that $\verti{f}_{W^{s,p}} \neq 0$. 
    \begin{enumerate}[(a)]
        \item If $s$ non-integer is such that $sp<N$, $(f_n) \subset \mathring{W}^{s,p}$, $f \in \mathring{W}^{s,p}$, and $f_n \rightharpoonup f$ in $\mathring{W}^{s,p}$, then 
        \bes
         \liminf_{n} \inf_{\xi \in \mathbb{S}^{N-1}} \int_{0}^{\infty} t^{-sp-1} \vertii{\Delta^{\lfloor s \rfloor +1}_{t\xi} f_n}_{L^p}^p \ dt >0.
         \ees
          \item If $s$ is an integer and $f_n \rightharpoonup f$ in $\dot{W}^{s,p}$, then
        \bes \liminf_{n}\inf_{\xi \in \mathbb{S}^{N-1}} \vertii{\partial^s_{\xi} f_n}_{L^p} >0.
        \ees
    \end{enumerate}
   
\end{lemma}

\begin{proof}
(a) Let $(f_n)$ be such that $ f_n \rightharpoonup f$ in $\mathring{W}^{s,p}$ and \bes \liminf_n \inf_{ \ \xi \in \mathbb{S}^{N-1}} \int_{0}^{\infty} t^{-sp-1} \vertii{\Delta^{\lfloor s \rfloor +1}_{t\xi} f_n}_{L^p}^p \ dt=0. \ees We find sequences $(\xi_k) \subset \mathbb{S}^{N-1}$ and $(n_k) \subset \N$ such that
   \bes \label{strong conv}
     \vertii{F_{n_k, \xi_k}}^p_{L^p(\R^{+}\times \R^N)}=\int_{0}^{\infty} t^{-sp-1} \vertii{\Delta^{\lfloor s \rfloor +1}_{t\xi_k} f_{n_k}}_{L^p}^p \ dt \to 0,
   \ees
   and $\xi_k \to \xi$, as $k \to \infty$. (Here we use the notation $F_{n,\xi}$ introduced in Lemma \ref{dir weak conv}).  
   
   By Lemma \ref{dir weak conv}, we also have $F_{n_k,\xi_k} \to F_{\xi}$ in $\mathscr{D'}((0,\infty) \times \R^N)$,  hence
   \be \label{D lsc}
   \int_{0}^{\infty} t^{-sp-1} \vertii{\Delta^{\lfloor s \rfloor +1}_{t\xi} f}_{L^p}^p \ dt= \vertii{F_{\xi}}^p_{L^p(\R^{+}\times \R^N)}  \leq \liminf_k \vertii{F_{n_k, \xi_k}}^p_{L^p(\R^{+}\times \R^N)}=0.
   \ee 
   Therefore, we find that $\verti{f}_{W^{s,p}}=0$ by Lemma \ref{no constant direction}.

\noindent (b) We argue as in (a), using Lemma \ref{weak int conv} instead of Lemma \ref{dir weak conv}.
\end{proof}
As a simple consequence of Lemma \ref{far from 0}, we have the following result.
\begin{lemma}\label{weak lsc}
   Let $s>0$, $1 \leq p < \infty$, $(f_n) \subset \dot{W}^{s,p}$, and $f \in \dot{W}^{s,p}$ be such that $\verti{f}_{W^{s,p}} \neq 0$. 
    \begin{enumerate}[(a)]
        \item If $s$ non-integer is such that $sp<N$, $(f_n) \subset \mathring{W}^{s,p}$, $f \in \mathring{W}^{s,p}$, and $f_n \rightharpoonup f$ in $\mathring{W}^{s,p}$, then 
        \bes
         \mathscr{E}_{s,p}(f) \leq \liminf_n \mathscr{E}_{s,p}(f_n).
         \ees
          \item The same conclusion holds if $s$ is an integer and $f_n \rightharpoonup f$ in $\dot{W}^{s,p}$, then
          \bes
           \mathscr{E}_{s,p}(f) \leq \liminf_n \mathscr{E}_{s,p}(f_n).
          \ees
    \end{enumerate}
\end{lemma}
\begin{proof}

(a) We have, as in \eqref{D lsc}, 
\be \label{2 lsc}
 \int_{0}^{\infty} t^{-sp-1} \vertii{\Delta^{\lfloor s \rfloor +1}_{t\xi} f}_{L^p}^p \ dt \leq \liminf_n  \int_{0}^{\infty} t^{-sp-1} \vertii{\Delta^{\lfloor s \rfloor +1}_{t\xi} f_n}_{L^p}^p \ dt, \ \fo \xi \in \mathbb{S}^{N-1}. 
\ee 

But $(f_n)$ also satisfies the conclusion of Lemma \ref{far from 0} (a), hence
\bes
\left(\int_{0}^{\infty} t^{-sp-1} \vertii{\Delta^{\lfloor s \rfloor +1}_{t\xi} f_n}_{L^p}^p \ dt \right)^{-N/sp} \leq C <\infty, \ \fo \xi \in \mathbb{S}^{N-1},
\ees
for each sufficiently large $n \in \N$. 
Therefore, Lemma \ref{rev fatou} yields
\bes
\ba
\limsup_n \int_{\mathbb{S}^{N-1}}& \left( \int_{0}^{\infty} t^{-sp-1} \vertii{\Delta^{\lfloor s \rfloor +1}_{t\xi} f_n}_{L^p}^p \ dt\right)^{-N/sp} \ d \mathscr{H}^{N-1}(\xi) \\ & \leq \int_{\mathbb{S}^{N-1}} \limsup_n \left( \int_{0}^{\infty} t^{-sp-1} \vertii{\Delta^{\lfloor s \rfloor +1}_{t\xi} f_n}_{L^p}^p \ dt\right)^{-N/sp} \ d \mathscr{H}^{N-1}(\xi)\\ & = \int_{\mathbb{S}^{N-1}} \left(\liminf_n \int_{0}^{\infty} t^{-sp-1} \vertii{\Delta^{\lfloor s \rfloor +1}_{t\xi} f_n}_{L^p}^p \ dt\right)^{-N/sp} \ d \mathscr{H}^{N-1}(\xi)
\\ & \leq \int_{\mathbb{S}^{N-1}} \left(\int_{0}^{\infty} t^{-sp-1} \vertii{\Delta^{\lfloor s \rfloor +1}_{t\xi} f}_{L^p}^p \ dt\right)^{-N/sp} \ d \mathscr{H}^{N-1}(\xi)
\ea
\ees
where we rely on \eqref{2 lsc} to obtain the last inequality.
Raising the last inequality to the $-sp/N$ power, we obtain the desired conclusion. The proof of (b) is similar. \qedhere

\end{proof}\begin{lemma}\label{not in 0}
Let $s>0$ and $1<p<\infty$.  Let $(f_n) \subset \dot{W}^{s,p}$ be such that $f_n \rightharpoonup f \neq 0 $ in $\dot{W}^{s,p}$ and $\mathscr{E}_{s,p}(f_n) \to \mathscr{E}_{s,p}(f) $, as $n \to \infty$.
\begin{enumerate}[(a)]
    
    \item If $s$ is non-integer such that $sp<N$ and $f_n \in \mathring{W}^{s,p}$, then $f_n \to f$ in $\dot{W}^{s,p}$. 
    \item If $s$ is an integer, then $f_n \to f $ in $\dot{W}^{s,p}$.
\end{enumerate}
\end{lemma}
\begin{proof} We only prove (b) ((a) is similar). We first show that $\vertii{\partial^{s}_{\xi} f_n}_{L^p} \to \vertii{\partial_{\xi}^{s} f}_{L^p}$, as $n\to \infty$, for each $\xi \in \mathbb{S}^{N-1}$. 
    
    By contradiction, assume that there exists $\xi \in \mathbb{S}^{N-1}$ and $\varepsilon>0$ such that $\vertii{\partial_{\xi}^{s} f_n}_{L^p} \geq \vertii{\partial_{\xi}^{s} f}_{L^p} +\varepsilon $, for each sufficiently large $n$. By \eqref{Lip}, this implies that there exists a set $A \subset \mathbb{S}^{N-1}$ such that $\mathscr{H}^{N-1}(A) >0$, and $\vertii{\partial^{s}_{\omega}f_n}_{L^p} \geq \vertii{\partial^{s}_{\omega}f}_{L^p} +\varepsilon/2$, for each $\omega \in A$, for each $n$ sufficiently large.  We then have
    \be\label{a}
    \ba
     \int_{\mathbb{S}^{N-1}} \vertii{\partial^s_{\xi} f_n}_{L^p}^{-N/s} \ d\mathscr{H}^{N-1}(\xi) \leq& \int_{A} \left(\vertii{\partial_ {\xi}^sf}_{L^p}+ \varepsilon/2\right)^{-N/s} \ d\mathscr{H}^{N-1}(\xi) 
     \\ & +   \int_{\mathbb{S}^{N-1}\setminus A} \vertii{\partial_ {\xi}^s f_n}_{L^p}^{-N/s} \ d\mathscr{H}^{N-1}(\xi),
    \ea
    \ee
    for each sufficiently large $n$.
    By Lemma \ref{far from 0}, we have
     \bes \label{lower}
     \vertii{\partial_{\xi}^s f_n}_{L^p}^{-N/s} \leq C <\infty, \ \fo \xi \in \mathbb{S}^{N-1},
    \ees
    for each sufficiently large $n \in \N$. 
    Therefore, Lemma \ref{rev fatou} yields
    \bes
     \limsup_n \int_{\mathbb{S}^{N-1}\setminus A} \vertii{\partial_ {\xi}^s f_n}_{L^p}^{-N/s} \ d\mathscr{H}^{N-1}(\xi) \leq \int_{\mathbb{S}^{N-1}\setminus A} \limsup_n \vertii{\partial_ {\xi}^s f_n}_{L^p}^{-N/s} \ d\mathscr{H}^{N-1}(\xi).
    \ees
    Passing to the $\limsup$ in \eqref{a}, we find
    \bes
    \ba
     &\mathscr{E}_{s,p}(f)^{-N/s} \\ &\leq \int_{A} \left(\vertii{\partial_ {\xi}^sf}_{L^p}+ \varepsilon/2\right)^{-N/s} \ d\mathscr{H}^{N-1}(\xi)
     + \limsup_n \int_{\mathbb{S}^{N-1}\setminus A} \vertii{\partial_ {\xi}^s f_n}_{L^p}^{-N/s} \ d\mathscr{H}^{N-1}(\xi)
     \\ &\leq  \int_{A} \left(\vertii{\partial_ {\xi}^sf}_{L^p}+ \varepsilon/2\right)^{-N/s} \ d\mathscr{H}^{N-1}(\xi)
      +\int_{\mathbb{S}^{N-1}\setminus A} \limsup_n \vertii{\partial_ {\xi}^s f_n}_{L^p}^{-N/s} \ d\mathscr{H}^{N-1}(\xi)
      \\ & \leq  \int_{A} \left(\vertii{\partial_ {\xi}^sf}_{L^p}+ \varepsilon/2\right)^{-N/s} \ d\mathscr{H}^{N-1}(\xi)
      +\int_{\mathbb{S}^{N-1}\setminus A}  \vertii{\partial_ {\xi}^s f}_{L^p}^{-N/s} \ d\mathscr{H}^{N-1}(\xi)
      \\ & < \mathscr{E}_{s,p}(f)^{-N/s}
    \ea
    \ees
    and a contradiction. 
    
    Hence, for each $\xi \in \mathbb{S}^{N-1}$, we have $\partial^s_{\xi} f_n \rightharpoonup \partial^s_{\xi} f$ in $L^p$ and $\vertii{\partial^s_{\xi} f_n }_{L^p} \to \vertii{\partial^s_{\xi}f}_{L^p}$. Since $L^p$ with $1<p<\infty$ is uniformly convex, we find that $\partial^s_{\xi} f_n \to \partial^s_{\xi} f$ in $L^p$ , as $n \to \infty$. This clearly implies that $\verti{f_n - f}_{W^{s,p}} \to 0$, as $n \to \infty$.
\end{proof}
\begin{rema}
    Lemma \ref{not in 0} is in contrast with the fact that there exist sequences $(f_n) \subset \dot{W}^{1,p}$ such that $\vertii{\nabla f_n}_{L^p}=1$, for each $n \in \N$, and $\mathscr{E}_{1,p}(f_n) \to 0$, $f_n \rightharpoonup 0$ in $\dot{W}^{1,p}$, see \cite[Theorem 2]{haddad2021affine}. 
\end{rema}

\section{Proof of Theorem \ref{exist sob} in the fractional case}\label{frac}
We start this section with a key auxiliary result.
\begin{lemma}\label{uniform BL}
    Let $s$ be non-integer and $1 \leq p < \infty$ be such that $sp<N$. If $f_n \rightharpoonup f$ in $\mathring{W}^{s,p}$, then 
    \bes
    \ba
    \sup_{\xi \in \mathbb{S}^{N-1}} \int_{0}^{\infty}\int_{\R^N} t^{-sp-1}\bigg|\verti{\Delta_{t\xi}^{\lfloor s \rfloor +1} f_n(x)}^p - &\verti{\Delta_{t\xi}^{\lfloor s \rfloor +1}(f_n(x)-f(x))}^p- \verti{\Delta_{t\xi}^{\lfloor s \rfloor +1} f(x)}^p  \bigg| \ dx dt \\ & \longrightarrow 0, \ \text{as} \ n \to \infty.
    \ea
    \ees
\end{lemma}
In the proof of Lemma \ref{uniform BL}, we will rely on the following results. 
\begin{lemma}\label{equi}
Let $s$ be non-integer and $1 \leq  p< \infty$ be such that $sp<N$. Let $M>0$, $\delta>0$, and $K \subset \R^N$ be a compact set.
We have
\bes
\lim_{\varepsilon \to 0} \underset{\substack{f \in \mathring{W}_M^{s,p} \\  \omega, \ \xi \in \  \mathbb{S}^{N-1}, \    |\omega-\xi| \leq \varepsilon}}{\sup}\ \verti{\int_{\delta}^{\infty} t^{-sp-1}\vertii{\Delta_{t\omega}^{\lfloor s \rfloor +1} f}_{L^p(K)}^p \ dt -  \int_{\delta}^{\infty}t^{-sp-1} \vertii{\Delta_{t\xi}^{\lfloor s \rfloor +1} f}_{L^p(K)}^p \ dt}  = 0, 
\ees
where $\mathring{W}_M^{s,p} \coloneq \{ f \in \mathring{W}^{s,p}; \, \verti{f}_{W^{s,p}} \leq M \}$.
\end{lemma}
\begin{lemma}\label{c0}
    Let $s$ be non-integer and $1\leq p<\infty$ be such that $sp<N$. For each $f \in \mathring{W}^{s,p}$ and each measurable set $B \subset (0,\infty) \times \R^N$,
    \bes
     \mathbb{S}^{N-1} \ni \xi \to \int_{B} t^{-sp-1} \verti{\Delta^{\lfloor s \rfloor +1}_{t\xi} f(x)}^p \ dx dt
    \ees
    is continuous.
\end{lemma}
Granted Lemmas \ref{equi} and \ref{c0}, we turn to the

\begin{proof}[Proof of Lemma \ref{uniform BL}]
It suffices to show that, for each sequence $(\xi_n) \subset \mathbb{S}^{N-1}$, we have
 \bes
\ba
\int_{0}^{\infty}\int_{\R^N} t^{-sp-1}&\verti{\verti{\Delta_{t\xi_n}^{\lfloor s \rfloor +1} f_{n}(x)}^p - \verti{\Delta_{t\xi_n}^{\lfloor s \rfloor +1}(f_{n}(x)-f(x))}^p-\verti{\Delta_{t\xi_n}^{\lfloor s \rfloor +1} f(x)}^p } \ dx dt \\ &\longrightarrow 0, \ \text{as} \ n \to \infty.
    \ea
    \ees
    Without loss of generality, we may assume that $\xi_n \to \xi \in \mathbb{S}^{N-1}$ and that $f_n \to f$ a.e.

We first show that, for each $R, \delta>0$,
\be\label{1st goal}
\int_{\delta}^{\infty}\int_{B(0,R)} t^{-sp-1}G_n(\xi_n,t,x)\ dx dt \to 0, \ \text{as} \ n \to \infty,
\ee
where
\bes
G_n(\omega,t,x)\coloneq \verti{\verti{\Delta_{t\omega}^{\lfloor s \rfloor +1} f_{n}(x)}^p - \verti{\Delta_{t\omega}^{\lfloor s \rfloor +1}(f_{n}(x)-f(x))}^p-\verti{\Delta_{t\omega}^{\lfloor s \rfloor +1} f(x)}^p }, 
\ees
for each $\omega \in \mathbb{S}^{N-1},t>0, x \in \R^N$. 

On the one hand, we have
 \be
    \ba
    \label{in xi}
    \int_{\delta}^{\infty}\int_{B(0,R)} t^{-sp-1}G_n(\xi,t,x) \ dx dt  \longrightarrow 0, \ \text{as} \ n \to \infty,
    \ea
    \ee
by an application of the Brezis-Lieb lemma to the sequence
\bes
(\delta,\infty) \times B(0,R) \ni (t,x) \mapsto t^{-s-1/p}\Delta_{t\xi}^{\lfloor s \rfloor +1} f_n(x),
\ees
which is bounded in $L^p((0,\infty)\times \R^N)$ (by Lemma \ref{dir domination}) and converges almost everywhere to the function~ $(t,x) \mapsto  t^{-sp-1}\Delta_{t\xi}^{\lfloor s \rfloor +1} f(x)$.

On the other hand, by Lemma \ref{equi} and the triangular inequality, we have
\be \label{uniform l}
\ba
 \underset{\substack{n \in \N \\ \omega \in \mathbb{S}^{N-1}, \verti{\omega- \xi} \leq \varepsilon}}{\sup}  \biggl| \int_{\delta}^{\infty}\int_{B(0,R)}t^{-sp-1}G_n(w,t,x)  \ dx dt &- \int_{\delta}^{\infty}\int_{B(0,R)}t^{-sp-1}G_n(\xi,t,x)  \ dx dt\biggl|
 \\ &  \longrightarrow 0, \ \text{as} \ \varepsilon \to 0.
 \ea
 \ee 
 Since $\xi_n \to \xi$, combining \eqref{in xi} and \eqref{uniform l}, we obtain \eqref{1st goal}.

 We now complete the proof by showing that, for each $\varepsilon>0$, there exist $R,\delta>0$, such that 
 \bes
 \ba
 &\int_{\left((\delta, \infty)\times B(0,R) \right)^{C}} t^{-sp-1}G_n(\xi_n,t,x)\ dx dt <\varepsilon,
 \ea
 \ees
 for each sufficiently large $n$.
 
 By Lemma \ref{dir domination}, we have
\bes
 \alpha \coloneq \sup_n \int_{0}^{\infty}\int_{\R^N} t^{-sp-1} \verti{\Delta^{\lfloor s \rfloor +1}_{t\xi_n} (f_ {n}(x)-f(x))}^p \ dx dt <\infty.
\ees
We may therefore fix 
\be \label{eta ch}
0<\eta  < \frac{\varepsilon}{3 \alpha}.  
\ee
    
There exists $C(\eta)<\infty$ such that
\be\label{mink eps}
\verti{\verti{a+b}^p - \verti{b}^p} \leq \eta \verti{b}^p +C(\eta) \verti{a}^p, \ \fo a,b \in \R 
\ee
and this yields
\be \label{eta est}
\ba
&\int_{B} t^{-sp-1} G_n(\xi_n,t,x) \ dx dt \\ &=\int_{B} t^{-sp-1}\verti{\verti{\Delta_{t\xi_n}^{\lfloor s \rfloor +1} f_{n}(x)}^p - \verti{\Delta_{t\xi_n}^{\lfloor s \rfloor +1}(f_{n}(x)-f(x))}^p-\verti{\Delta_{t\xi_n}^{\lfloor s \rfloor +1} f(x)}^p } \ dx dt \\ & \leq \eta \int_{B} t^{-sp-1} \verti{\Delta^{\lfloor s \rfloor +1}_{t\xi_n} (f_{n}(x)-f(x))}^p \ dx dt   \\ &+ (C(\eta)+1) \int_{B} t^{-sp-1} \verti{\Delta^{\lfloor s \rfloor +1}_{t\xi_n} f(x)}^p \ dx dt
\\ & \leq \varepsilon/3 + (C(\eta)+1)\int_{B} t^{-sp-1} \verti{\Delta^{\lfloor s \rfloor +1}_{t\xi_n} f(x)}^p \ dx dt, 
\ea
\ee
for each measurable set $B \subset (0,\infty) \times \R^N$, using \eqref{mink eps} with
\bes
a=\Delta^{\lfloor s \rfloor +1}_{t\xi_n}f(x), \ b= \Delta^{\lfloor s \rfloor +1}_{t\xi_n}(f_n(x)-f(x)).
\ees
 Let $\delta>0$ be sufficiently small, $R$ be sufficiently large, such that
\be\label{LEBE LEMMA}
\ba
&(C(\eta)+1) \int_{0}^{\delta} \int_{\R^N} t^{-sp-1} \verti{\Delta^{\lfloor s \rfloor +1}_{t\xi} f(x)}^p \ dx dt  < \varepsilon/3,
\\ & (C(\eta)+1)\int_{0}^{\infty} \int_{\R^N\setminus B(0,R)} t^{-sp-1} \verti{\Delta^{\lfloor s \rfloor +1}_{t\xi} f(x)}^p \ dx dt <\varepsilon/3.
\ea
\ee
(The existence of such $\delta, R>0$ follows from dominated convergence).
Since $\xi_n \to \xi$, combining \eqref{LEBE LEMMA} and Lemma \ref{c0}, we obtain
\bes
\ba
&(C(\eta)+1) \int_{0}^{\delta} \int_{\R^N} t^{-sp-1} \verti{\Delta^{\lfloor s \rfloor +1}_{t\xi_n} f(x)}^p \ dx dt<\varepsilon/3 \\
&(C(\eta)+1) \int_{0}^{\infty} \int_{\R^N \setminus B(0,R)} t^{-sp-1} \verti{\Delta^{\lfloor s \rfloor +1}_{t\xi_n} f(x)}^p \ dx dt <\varepsilon/3,
\ea
\ees
i.e. 
\be \label{by c0}
(C(\eta)+1) \int_{\left((\delta,\infty) \times B(0,R) \right)^C} t^{-sp-1} \verti{\Delta^{\lfloor s \rfloor +1}_{t\xi_n} f(x)}^p \ dx dt<2\varepsilon/3, 
\ee
for each sufficiently large $n$.

In view of \eqref{eta ch}, \eqref{eta est}, and \eqref{by c0}, we find that, for this choice of $\delta,R>0$,
\bes
\ba
&\int_{\left((\delta, \infty) \times B(0,R) \right)^{C}} t^{-sp-1}G_n(\xi_n,t,x)\ dx dt 
\\ & \leq \varepsilon/3 + (C(\eta)+1) \int_{\left((\delta,\infty) \times B(0,R) \right)^C} t^{-sp-1} \verti{\Delta^{\lfloor s \rfloor +1}_{t\xi_n} f(x)}^p \ dx dt  < \varepsilon,
\ea
\ees
for each sufficiently large $n$.
\end{proof}

Lemma \ref{equi} is a consequence of the following
\begin{lemma}\label{norm conv cpt}
    Let $s$ be non-integer and $1 \leq p<\infty$ be such that $sp<N$. If $(f_n) \subset \mathring{W}^{s,p}$ weakly converges to $f$ in $\mathring{W}^{s,p}$, and $(\xi_n) \subset \mathbb{S}^{N-1}$ converges to $\xi$, then
    \be \label{norm limit}
    \lim_n \int_{\delta}^{\infty}t^{-sp-1}\vertii{\Delta^{\lfloor s \rfloor +1}_{t\xi_n}f_n}_{L^p(K)}^p \ dt = \int_{\delta}^{\infty}t^{-sp-1}\vertii{\Delta^{\lfloor s \rfloor +1}_{t\xi}f }_{L^p(K)}^p \ dt, 
    \ee
    for each $\delta>0$ and each compact set $K \subset \R^N$.
\end{lemma}
\begin{proof}
    Let $\delta >0$ and consider a compact set $K \subset \R^N$.  

      On the one hand, we have $\lim_n \vertii{\Delta^{\lfloor s \rfloor +1}_{t\xi_n}f_n}_{L^p(K)} = \vertii{\Delta^{\lfloor s \rfloor +1}_{t\xi}f}_{L^p(K)}$, for each $t>0$. Indeed, we have
      \bes
      \ba
      \biggl|& \vertii{\Delta^{\lfloor s \rfloor +1}_{t\xi_n}f_n}_{L^p(K)} - \vertii{\Delta^{\lfloor s \rfloor +1}_{t\xi}f}_{L^p(K)}\biggl| \\ &\leq C \sum_{k=0}^{\lfloor s \rfloor +1} \vertii{f_n(\cdot+kt\xi_n)-f(\cdot +kt\xi)}_{L^p(K)} 
      \\ & \leq C\sum_{k=0}^{\lfloor s \rfloor +1} \left(\vertii{f_n(\cdot + kt\xi_n)-f(\cdot+kt\xi_n)}_{L^p(K)} + \vertii{f(\cdot + kt\xi_n)-f(\cdot+kt\xi)}_{L^p(K)} \right).
      \ea
      \ees
      For each $0 \leq k\leq \lfloor s \rfloor +1$, $\lim_n \vertii{f_n(\cdot+kt\xi_n)-f(\cdot +kt\xi_n)}_{L^p(K)} = 0$, since $f_n \to f$ in $L^p_{\text{loc}}$, and $ \lim_n \vertii{f(\cdot + kt\xi_n)-f(\cdot+kt\xi)}_{L^p(K)}=0$, by continuity of translations in $L^p$ and since $\xi_n \to \xi $. 

On the other hand, for each $n \in \N$ and $t>0$, we have
\bes
\ba
\vertii{\Delta^{\lfloor s \rfloor +1}_{t\xi_n}f_n }_{L^p(K)} &\leq C \sum_{k=0}^{\lfloor s \rfloor +1} \vertii{f_n(\cdot+kt\xi_n)}_{L^p(K)} 
\\ &\leq  C \sum_{k=0}^{\lfloor s \rfloor +1} \verti{K}^{1/p-1/q} \vertii{f_n}_{L^q} \leq  C\verti{K}^{1/p-1/q} \verti{f_n}_{W^{s,p}}.
\ea
\ees
Here, we rely on H\"older's inequality to obtain the second inequality, and on the Sobolev embedding to obtain the last one. Since $(\verti{f_n}_{W^{s,p}})$ is bounded, the last estimate yields
\bes
t^{-sp-1} \vertii{\Delta^{\lfloor s \rfloor +1}_{t\xi_n}f_n }_{L^p(K)} \leq Ct^{-sp-1},
\ees
for some $C$ which is independent of $t$ and $n$. Since $t \mapsto t^{-sp-1}$ is in $L^1(\delta,\infty)$, \eqref{norm limit} follows by dominated convergence.
\end{proof}
\begin{proof}[Proof of Lemma \ref{equi}]
By contradiction, assume that there exist $(\omega_n), (\xi_n) \subset \mathbb{S}^{N-1}$ with $\verti{\xi_n -\omega_n} \leq 1/n$, and $(f_n) \subset \mathring{W}_M^{s,p}$, such that 
\bes
\liminf_n \verti{\int_{\delta}^{\infty} t^{-sp-1}\vertii{\Delta_{t\omega_n}^{\lfloor s \rfloor +1} f_n}_{L^p(K)}^p \ dt -  \int_{\delta}^{\infty} t^{-sp-1}\vertii{\Delta_{t\xi_n}^{\lfloor s \rfloor +1} f_n}_{L^p(K)}^p \ dt}>0.
\ees
We may extract subsequences $(\xi_{n_k})$, $(\omega_{n_k})$, and $(f_{n_k})$ such that $\omega_{n_k} \to \xi$, $\xi_{n_k}\to \xi$ (since $\verti{\omega_n - \xi_n} \leq 1/n$) , and $f_{n_k} \rightharpoonup f$ in $\mathring{W}^{s,p}$ (by Lemma \ref{weak comp}). By Lemma \ref{norm conv cpt}, we obtain 
and a contradiction.
\end{proof}

\begin{proof}[Proof of Lemma \ref{c0}]
We first consider the case where $f \in \mathscr{S}$. 

Let $(\xi_n) \subset \mathbb{S}^{N-1}$ converge to $\xi$ and let $\varepsilon>0$. 
Since $f\in \mathscr{S}$, we have
\bes
\int_{0}^{\infty}t^{-sp-1} \sup_{\verti{h}\leq t} \vertii{\Delta^{\lfloor s \rfloor +1}_{h} f}^p_{L^p} \ dt < \infty.
\ees
Thus, by dominated convergence, we may find $0<a<b$ such that 
\bes
\int_{0}^{a}t^{-sp-1} \vertii{\Delta^{\lfloor s \rfloor +1}_{t\omega} f}^p_{L^p} \ dt <\varepsilon, \int_{b}^{\infty}t^{-sp-1} \vertii{\Delta^{\lfloor s \rfloor +1}_{t\omega} f}^p_{L^p} \ dt <\varepsilon, \ \fo \omega \in \mathbb{S}^{N-1}.
\ees
Therefore, we have
\bes
\ba
&\verti{\int_{B}t^{-sp-1}\verti{\Delta^{\lfloor s \rfloor +1}_{t\xi_n} f(x)}^p\ dx dt - \int_{B}t^{-sp-1}\verti{\Delta^{\lfloor s \rfloor +1}_{t\xi} f(x)}^p \ dx dt} \\  &\leq 4 \varepsilon +    \verti{\int_{B\cap \left((a,b)\times \R^N\right)}t^{-sp-1}\verti{\Delta^{\lfloor s \rfloor +1}_{t\xi_n} f(x)}^p \ dx dt - \int_{B\cap \left((a,b)\times \R^N\right)} t^{-sp-1}\verti{\Delta^{\lfloor s \rfloor +1}_{t\xi} f(x)}^p \ dxdt}
\ea
\ees
and it suffices to show that 
\be \label{ab}
\ba
\int_{B\cap \left((a,b)\times \R^N\right)}t^{-sp-1}&\verti{\Delta^{\lfloor s \rfloor +1}_{t\xi_n} f(x)}^p \ dx dt \to \int_{B\cap \left((a,b)\times \R^N\right)}t^{-sp-1}\verti{\Delta^{\lfloor s \rfloor +1}_{t\xi} f(x)}^p \ dx dt,
\ea
\ee
as $n \to \infty$,
to complete the proof of Lemma \ref{c0} in the case where $f \in \mathscr{S}$.

We have 
\bes
t^{-sp-1} \Delta_{t\xi_n }^{\lfloor s \rfloor+1} f(x) \to t^{-sp-1} \Delta_{t\xi}^{\lfloor s \rfloor+1} f(x), \ \fo (t,x) \in (0,\infty) \times \R^N.
\ees
We also have 
\be \label{S domination}
\ba
 t^{-sp-1} \verti{\Delta_{t\xi_n }^{\lfloor s \rfloor+1} f(x)}^p &\leq t^{\lfloor s \rfloor p+p-sp-1} \left(\int_{0}^{\lfloor s \rfloor +1} \vertii{D^{\lfloor s \rfloor+1} f(x+u\xi_n)}  \ du \right)^{p}
 \\ & \leq Ct^{\lfloor s \rfloor p+p-sp-1}\sup_{\verti{y-x}\leq \lfloor s \rfloor +1} \vertii{D^{\lfloor s \rfloor +1}f(y)}^p, 
\ea 
\ee
for each $(t,x) \in (0,\infty) \times\R^N$. 
Since $f \in \mathscr{S}$, we have
\be \label{x est}
\ba
\sup_{\verti{y-x}\leq \lfloor s \rfloor +1} \vertii{D^{\lfloor s \rfloor +1}f(y)} &\leq C \sup_{\verti{y-x}\leq \lfloor s \rfloor +1} \left(\lfloor s \rfloor +2+\verti{y}\right)^{-N-1} \\ & \leq C  \left(1 + \verti{x} \right)^{-N-1}, \ \fo x \in \R^N.
\ea
\ee
\eqref{x est} shows that the right-hand side in \eqref{S domination} is $L^1((a,b) \times \R^N)$. 
We may therefore apply dominated convergence to obtain \eqref{ab}.

If $f \in \mathring{W}^{s,p}$, we find $(f_n) \subset \mathscr{S}$ such that $\verti{f_n-f}_{W^{s,p}} \to 0$ (by Lemma \ref{density}). For each $n \in \N$, $\xi \in \mathbb{S}^{N-1}$, we have
\bes
\ba
&\verti{\int_{B} t^{-sp-1}\verti{\Delta^{\lfloor s \rfloor +1}_{t\xi} f_n(x)}^p \ dx dt -  \int_{B} t^{-sp-1}\verti{\Delta^{\lfloor s \rfloor +1}_{t\xi} f(x)}^p\ dx dt}
\\ & \leq C \int_{B} t^{-sp-1} \verti{\Delta^{\lfloor s \rfloor +1}_{t\xi} (f_n-f)(x)} \left( \verti{\Delta^{\lfloor s \rfloor +1}_{t\xi} f_n(x)}^{p-1} +   \verti{\Delta^{\lfloor s \rfloor +1}_{t\xi} f(x)}^{p-1} \right)\ dx dt  
\\ & \leq C \int_{0}^{\infty} t^{-sp-1} \int_{\R^N} \verti{\Delta^{\lfloor s \rfloor +1}_{t\xi} (f_n-f)(x)} \left( \verti{\Delta^{\lfloor s \rfloor +1}_{t\xi} f_n(x)}^{p-1} +   \verti{\Delta^{\lfloor s \rfloor +1}_{t\xi} f(x)}^{p-1} \right)\ dx dt 
\\ & \leq C \left( \int_{0}^{\infty} t^{-sp-1}  \vertii{\Delta^{\lfloor s \rfloor +1}_{t\xi} (f_n-f)}_{L^p}^p \ dt \right)^{1/p} \\
& \left(\int_{0}^{\infty} t^{-sp-1}  \left(\vertii{\Delta^{\lfloor s \rfloor +1}_{t\xi}f}^p_{L^p} + \vertii{\Delta^{\lfloor s \rfloor +1}_{t\xi}f_n}^p_{L^p} \right)\ dt \right)^{1-1/p}
\\ &\leq C \verti{f_n-f}_{W^{s,p}}, 
\ea
\ees
where we used the mean value inequality to obtain the first inequality, H\"older's inequality to obtain the third one, and Lemma \ref{dir domination} and the boundedness of $(f_n)$ in $\dot{W}^{s,p}$ for the last one. Therefore, we have

\be \label{unif conv}
\ba
\sup_{\xi \in \mathbb{S}^{N-1}}&\verti{\int_{B}t^{-sp-1}\verti{\Delta^{\lfloor s \rfloor +1}_{t\xi} f_n(x)}^p \ dt -  \int_{B}t^{-sp-1}\verti{\Delta^{\lfloor s \rfloor +1}_{t\xi} f(x)}^p\  dxdt }
\\ & \leq C \verti{f_n-f}_{W^{s,p}} \to 0, \ \, \text{as} \ \, n \to \infty.
\ea
\ee
Since the maps $ \displaystyle \mathbb{S}^{N-1} \ni \xi \mapsto \int_{B} t^{-sp-1}\verti{\Delta^{\lfloor s \rfloor +1}_{t\xi} f_n(x)}^p\ \ dx \ dt $ are continuous by our first step, \eqref{unif conv} yields the desired conclusion.
\end{proof}

Thanks to Lemma \ref{uniform BL}, we are in position to adapt the approach of Zhang \cite{zhang2021optimizers} to the existence of extremal functions in the fractional Sobolev inequality, based on the following result.
\begin{lemma}\label{non zero ex}
Let $s$ be a non-integer and $1\leq p<\infty$ be such that $sp<N$.  
    If $(f_n) $ is a bounded sequence in $\mathring{W}^{s,p}$  such that $\inf_n \vertii{f_n}_{L^q} >0$, then there exist sequences $(R_n) \subset (0,\infty)$ and $(x_n) \subset \R^N$ such that, up to a subsequence, $\frac{1}{R_n^{N/q}}f_n((\cdot-x_n)/R_n) \rightharpoonup f \neq 0$ in $\mathring{W}^{s,p}$.
\end{lemma}
Lemma \ref{non zero ex} is a consequence of the improved Sobolev inequality 
\be\label{improved}
C \vertii{f}_{L^q} \leq  \verti{f}_{B^{-N/q}_{\infty,\infty}}^{sp/N} \verti{f}_{W^{s,p}}^{1-sp/N}, \ \fo f \in \mathring{W}^{s,p},
\ee
where $\verti{f}_{B^{-N/q}_{\infty,\infty}} \coloneq \sup_{A>0}A^{-N/q} \vertii{\mathscr{F}^{-1}\left(\Psi(A^{-1} \cdot) \mathscr{F}f\right)}_{L^{\infty}}$, with $0 \leq \Psi \leq 1$ compactly supported, and with value $1$ near $0$. A proof of \eqref{improved} may be found in \cite[Theorem 2.42]{bahouri2011fourier}, as well as the definition of the Besov space $B_{\infty,\infty}^{-N/sp}$ \cite[Definition 1.40]{bahouri2011fourier}. A proof of Lemma \ref{non zero ex}, when $0<s<1$ and $1<p<\infty$, may be found in \cite[Lemma 3.2]{zhang2021optimizers}. The same proof yields the general case, see Appendix \ref{A}.

We now turn to the 

\begin{proof}[Proof of Theorem \ref{exist sob} in the case where $s$ is non-integer]
Consider a sequence $(f_n) \subset \mathring{W}^{s,p}$ such that $\vertii{f_n}_{L^q}=1$ and $\mathscr{E}_{s,p}(f_n) \to S$.

 By Lemma \ref{equili}, replacing $f_n$ with $f_n\circ T_{f_n}$, we may assume that 
\be  \label{replace}
\verti{f_n}_{W^{s,p}}^p\leq C \int_{0}^{\infty}t^{-sp-1}\vertii{\Delta^{\lfloor s \rfloor +1}_{t\xi}f_n}_{L^p}^p\ dt, \ \fo \xi \in \mathbb{S}^{N-1}.
\ee
In particular, we may assume that $(f_n)$ is bounded in $\dot{W}^{s,p}$.
    By Lemma \ref{non zero ex}, there exist $(x_n)\subset \R^N$, $(R_n) \subset (0,\infty)$, and $f \neq 0$, such that
    \be \label{def til}
    \tilde{f_n}\coloneq \frac{1}{R_n^{N/q}}f_n((\cdot-x_n)/R_n) \rightharpoonup f \ \text{in} \ \mathring{W}^{s,p}
    \ee 
    up to a subsequence. 
    We may replace  $f_n$ by $\tilde{f_n}$ thanks to the identities $\vertii{\tilde{f_n}}_{L^q}= \vertii{f_n}_{L^q}$, 
\bes
\int_{0}^{\infty}t^{-sp-1}\vertii{\Delta^{\lfloor s \rfloor +1}_{t\xi}\tilde{f_n}}_{L^p}^p\ dt = \int_{0}^{\infty}t^{-sp-1}\vertii{\Delta^{\lfloor s \rfloor +1}_{t\xi}f_n}_{L^p}^p\ dt, \ \fo \xi \in \mathbb{S}^{N-1},  
\ees
and $\displaystyle| \tilde{f_n}|_{W^{s,p}}=\verti{f_n}_{W^{s,p}}$ for each $n \in \N$. In particular, $\tilde{f_n}$ satisfies \eqref{replace}. Hence, we may assume that $f_n \rightharpoonup f \neq 0$ in $\mathring{W}^{s,p}$.

    
    By Lemma \ref{uniform BL}, we have, for each $\xi \in \mathbb{S}^{N-1}$,
    \bes
    \ba
    \int_{0}^{\infty}t^{-sp-1} &\vertii{\Delta^{\lfloor s \rfloor +1}_{t\xi} f_n }_{L^p}^p \ dt \\& = \int_{0}^{\infty}t^{-sp-1} \vertii{\Delta^{\lfloor s \rfloor +1}_{t\xi} (f_n-f) }_{L^p}^p \ dt +  \int_{0}^{\infty}t^{-sp-1} \vertii{\Delta^{\lfloor s \rfloor +1}_{t\xi} f }_{L^p}^p \ dt + R_n(\xi)
    \\& = Q_n(\xi)+ R_n(\xi), 
    \ea
    \ees
    with $\sup_{\xi \in \mathbb{S}^{N-1}}\verti{R_n(\xi)} \to 0$, as $n \to \infty$. 
    Note that
     \be \label{for acc}
     \inf_{\xi  \in \mathbb{S}^{N-1}} Q_n(\xi) \geq \inf_{\xi \in \mathbb{S}^{N-1}} \int_{0}^{\infty}t^{-sp-1} \vertii{\Delta^{\lfloor s \rfloor +1}_{t\xi} f }_{L^p}^p \ dt >0.
    \ee
    The last inequality in \eqref{for acc} is a consequence of Lemma \ref{no constant direction}, since $\verti{f}_{W^{s,p}} \geq C\vertii{f}_{L^q}>0$, by the Sobolev inequality \eqref{sob}, and since $f\neq 0$.
    
    Therefore, we find that
     \be \label{affine BL}
    \ba
    &\mathscr{E}_{s,p}(f_n)^p \geq \bigg(\int_{\mathbb{S}^{N-1}}\left(Q_n(\xi)^{-N/sp} +  \varepsilon_n \right) \  d \mathscr{H}^{N-1}(\xi)\bigg)^{-sp/N} \\ 
    &\geq  \left(\int_{\mathbb{S}^{N-1}}Q_n(\xi)^{-N/sp} \ d \mathscr{H}^{N-1}(\xi) \right)^{-sp/N}  + o_{n\to \infty}(1)
    \\ & \geq \mathscr{E}_{s,p}(f_n-f)^p + \mathscr{E}_{s,p}(f)^p + o_{n\to \infty}(1),
    \ea
    \ee
    where $\varepsilon_n$ is independent of $\xi \in \mathbb{S}^{N-1}$ and converges to $0$. 
    Here, thanks to \eqref{for acc}, we apply the mean value inequality to the map $t \mapsto t^{-N/sp}$ to obtain the first inequality, and to $t \mapsto t^{-sp/N}$ to obtain the second one. We then use the reverse Minkowski inequality \eqref{reverse minko} to obtain the last estimate.

Therefore, combining \eqref{affine BL} and the affine Sobolev inequality, we find
\bes 
\ba
\mathscr{E}_{s,p}(f_n)^p \geq S^p\left(\vertii{f_n-f}_{L^q}^p + \vertii{f}_{L^q}^p\right) + o(1),
\ea
\ees
and thus
\bes
\ba
S^p=\lim_{n} \mathscr{E}_{s,p}(f_n)^p &\geq S^p  \left(\limsup_n \vertii{f_n -f}_{L^q}^p + \vertii{f}_{L^q}^p \right).
\ea
\ees
If $\limsup_n \vertii{f_n-f}_{L^q} >0$, we rely on the inequality
\be \label{n fact}
(a+b)^{p/q} < a^{p/q}+b^{p/q}, \ \fo a,b > 0,  \ \text{(since} \ p/q<1 )
\ee
and on the fact that $\vertii{f}_{L^q}>0$, to obtain 
\be \label{contra ineq}
\ba
S^p=\lim_{n} \mathscr{E}_{s,p}(f_n)^p &\geq S^p  \left(\limsup_n \vertii{f_n -f}_{L^q} ^p + \vertii{f}_{L^q}^p \right) \\ &>S^p \left(\limsup_n \vertii{f_n -f}_{L^q}^q + \vertii{f}_{L^q}^q \right)^{p/q}. 
\ea
\ee
By the Brezis-Lieb lemma, we have
\bes
\vertii{f_n-f}_{L^q}^q + \vertii{f}_{L^q}^q= \vertii{f_n}_{L^q}^q+o_{n\to \infty}(1)=1 + o(1)
\ees    
and this yields $\limsup_n \vertii{f_n-f}_{L^q}^q + \vertii{f}_{L^q}^q =1$. This fact, combined with \eqref{contra ineq}, yields a contradiction.

 Therefore, we find that $\vertii{f_n-f}_{L^q} \to 0$, and thus obtain that $\vertii{f}_{L^q}=1$. By the affine Sobolev inequality \eqref{aff sob} and Lemma \ref{weak lsc}, we hence have
 \bes
  S \leq \mathscr{E}_{s,p}(f) \leq \liminf_n \mathscr{E}_{s,p}(f_n) = S
 \ees
 i.e. $\mathscr{E}_{s,p}(f)=S$. Therefore, $f$ is an extremal function in the affine fractional Sobolev inequality. 
 
 When $p>1$, since $f_n \rightharpoonup f$ in $\mathring{W}^{s,p}$ and $\mathscr{E}_{s,p}(f_n) \to \mathscr{E}_{s,p}(f)$, we also find that $f_n \to f$ in $\dot{W}^{s,p}$, by Lemma \ref{not in 0}.
\end{proof}

\section{Proof of Theorem \ref{exist sob} in the integer case}\label{int}
Let $s$ be an integer and $1<p<\infty$.  We will rely on the following modification of the second concentration-compactness lemma. 
\begin{lemma}\label{affin CC}
    Let $(f_n) \subset \mathring{W}^{s,p}$ be such that $f_n \rightharpoonup f$ in $\dot{W}^{s,p}$, $\verti{f_n}^q \ dx \overset{*}{\rightharpoonup} \nu$, and
    \be \label{wH dir}
     \verti{\partial^s_{\xi}f_n}^p dx \overset{*}{\rightharpoonup} \verti{\partial^s_{\xi} f }^p dx + \mu_{\xi}, \  \fo \xi \in \mathbb{S}^{N-1}.
    \ee
    Then:
    \begin{enumerate}[(a)]
    \item
     $\displaystyle \nu= \verti{f}^q dx + \sum_{i \in I}\nu_i \delta_{x_i}$, where $I$ is countable, $(x_i)_{i \in I} \subset \R^N$, and $(\nu_i)_{i \in I} \subset (0,\infty) $.
    \item 
    $(\nu_i)_{i \in I}$ and $(\mu_{\xi})_{\xi \in \mathbb{S}^{N-1}}$ satisfy
    \bes
    \sum_{i}S^p(\nu_i)^{p/q} \leq \left(\int_{\mathbb{S}^{N-1}} \left( \sum_{i \in I} \mu_{\xi}(\{x_i\})\right)^{-N/sp} \ d\mathscr{H}^{N-1}(\xi) \right)^{-sp/N}. 
    \ees
    \end{enumerate}
\end{lemma}
Before turning to the proof of Lemma \ref{affin CC}, we first show that, from a bounded sequence in $\dot{W}^{s,p}$, we may always extract a subsequence satisfying the assumptions of Lemma \ref{affin CC}.
\begin{lemma}\label{extract}
    Let $(f_n)$ be  a bounded sequence in $\dot{W}^{s,p}$. Then there exists a subsequence $(f_{n_k})$ and a family of Radon measures $(\lambda_{\xi})_{\xi \in \mathbb{S}^{N-1}}$ such that, for each $\xi \in \mathbb{S}^{N-1}, \verti{\partial_{\xi}^s f_{n_k}}^p \ dx \overset{*}{\rightharpoonup} \lambda_{\xi},$ as $k\to \infty$.
    \end{lemma}
\begin{proof}
    Let $D$ be a dense countable subset of  $\mathbb{S}^{N-1}$. Using the compactness of bounded sequences of Radon measures, and a diagonal extraction procedure,  we may find a subsequence $(f_{n_k})$ and a family of finite measures $(\lambda_{\omega})_{\omega \in D}$ such that $\verti{\partial^s_{\omega}f_{n_k}}^p \ dx \ \overset{*}{\rightharpoonup} \lambda_{\omega}$, for each $\omega \in D$.

    We now show that $\verti{\partial_{\xi}^s f_{n_k}}^p \ dx$ converges in the sense of measures, for each $\xi \in \mathbb{S}^{N-1}$. It suffices to show that, for each $\varphi \in C_0(\R^N)$, the sequence $(\int_{\R^N} \varphi \verti{\partial^s_{\xi}f_{n_k}}^p \ dx)_k$ converges. Indeed, in that case, 
    \bes
    \varphi \mapsto  \lim_k \int_{\R^N} \varphi \verti{\partial^s_{\xi}f_{n_k}}^p \ dx 
    \ees
     is a continuous linear form on $C_0(\R^N)$, associated to a finite Radon measure $\lambda_{\xi}$ (by the Radon-Riesz theorem), and 
     $\verti{\partial^s_{\xi} f_{n_k}}^p \ dx \overset{*}{\rightharpoonup} \lambda_{\xi} $, as $k\to \infty$.

     Let $\varphi \in C_0(\R^N)$. For each $n \in \N$, $\eta, \eta' \in \mathbb{S}^{N-1}$, we note that 
     \be \label{unif}
     \ba
     \bigg|\int_{\R^N} \varphi\verti{\partial^s_{\eta}f_n}^p \ dx &- \int_{\R^N} \varphi\verti{\partial^s_{\eta'}f_n}^p \ dx \bigg| \\ &\leq p \vertii{\varphi}_{L^{\infty}} \int_{\R^N} \vertii{D_x^s f_n}^{p-1} \verti{ \partial^s_{\eta} f_n(x) - \partial^s_{\eta'}f_n(x)} \ dx 
     \\ & \leq pN  \vertii{\varphi}_{L^{\infty}}\verti{\eta - \eta'}\int_{\R^N} \vertii{D_x^s f_n}^{p}  \ dx
     \\ & \leq  C \verti{\eta -\eta'},
     \ea
     \ee
     where $C$ does not depend on $n$, since $(f_n)$ is bounded in $\dot{W}^{s,p}$.
     Here, we apply the mean value inequality to the map $t\mapsto t^p$ to obtain the first inequality. The second inequality is obtained as \eqref{Lip}.

     Now let $\varepsilon>0$. By density of $D$ and \eqref{unif}, we find $\omega \in D$ such that 
     \bes
        \verti{\int_{\R^N} \varphi\verti{\partial^s_{\xi}f_{n_k}}^p \ dx - \int_{\R^N} \varphi\verti{\partial^s_{\omega}f_{n_k}}^p \ dx } \leq \varepsilon, \ \fo k \in \N.
     \ees
     Since $\verti{\partial_{\omega}^s f_{n_k}}^p \ dx \overset{*}{\rightharpoonup} \lambda_{\omega}$,
    \bes
    \verti{\int_{\R^N} \varphi\verti{\partial^s_{\xi}f_{n_l}}^p \ dx - \int_{\R^N} \varphi\verti{\partial^s_{\xi}f_{n_k}}^p \ dx } \leq 3\varepsilon,
    \ees
    for each $k,l$ sufficiently large, by the triangular inequality. This completes the proof of Lemma \ref{extract}.
\end{proof}
We now turn to the proof of Lemma \ref{affin CC}, adapting the approach of Lions to the classical second concentration-compactness lemma (\cite[Lemma 1.1]{lions1985concentration}).
\begin{proof}[Proof of Lemma \ref{affin CC}]
(a) is given by the second concentration-compactness lemma \cite[Lemma 1.1]{lions1985concentration}.

For (b), we argue as follows. Note that $g_n\coloneq f_n-f$ is such that $g_n \rightharpoonup 0 \ \text{in} \ \dot{W}^{s,p}$. $(g_n)$ is bounded in $W^{s,p}(B(0,R))$, for each $R>0$, therefore, we may assume that $g_n \to 0$ in $W^{s-1,p}_{loc}$, by the Rellich-Kondrachov theorem. We may further assume that $g_n \to 0$ a.e. and obtain that 
    \be \label{mes limit}
     \verti{g_n}^q \ dx \overset{*}{\rightharpoonup} \tilde{\nu} \coloneq \sum_{i} \nu_i \delta_{x_i}, \ \text{as} \ n \to \infty,
    \ee 
 by the Brezis-Lieb lemma, and thanks to $(a)$.
 
    Let $\varphi \in C_c^{\infty}$. We have
    \be \label{sb est}
    S^p\left(\int_{\R^N} \verti{g_n\varphi}^q \ dx \right)^{p/q} \leq \mathscr{E}_{s,p}(g_n\varphi)^p, \ \fo n \in \N,
    \ee
    by definition of $S$, see \eqref{aff sob}.
    We also have
    \bes
    \ba
    &\mathscr{E}_{s,p}(g_n\varphi)^p= \left(\int_{\mathbb{S}^{N-1}}\vertii{\sum_{i=0}^s \binom{s}{i}\partial_{\xi}^i g_n \partial_{\xi}^{s-i}\varphi}_{L^p}^{-N/s} \ d\mathscr{H}^{N-1}(\xi) \right)^{-sp/N} 
    \\ & \leq  \left(\int_{\mathbb{S}^{N-1}}\left(\vertii{\partial^s_{\xi}f \varphi}_{L^p}+ \vertii{\partial^s_{\xi}f_n \varphi}_{L^p}+\vertii{\sum_{i=0}^{s-1}\binom{s}{i}\partial_{\xi}^i g_n \partial_{\xi}^{s-i}\varphi}_{L^p} \right)^{-N/s} \ d\mathscr{H}^{N-1}(\xi) \right)^{-sp/N}.
    \ea
    \ees
   Note that $Q_n(\xi) \coloneq \vertii{\sum_{i=0}^{s-1}\binom{s}{i}\partial_{\xi}^i g_n \partial_{\xi}^{s-i}\varphi}_{L^p}$ satisfies $Q_n(\xi) \to 0$, for each $\xi \in \mathbb{S}^{N-1}$, since $g_n \rightharpoonup 0$ in $W^{s-1,p}_{\text{loc}}$ and $\varphi \in C_c^{\infty}$. We also have that 
    \bes
    \vertii{\partial^s_{\xi}f_n \varphi}_{L^p} \to \left(\vertii{\partial^s_{\xi} f \varphi}_{L^p}^p +\int_{\R^N}\verti{\varphi}^p  d \mu_{\xi}   \right)^{1/p} ,
    \ees
    as $n \to \infty$, by \eqref{wH dir}.
    Therefore, we find that
    \begin{flalign}\begin{split} \label{lims}
    &\limsup_n \mathscr{E}_{s,p}(g_n\varphi)^p \\ &\leq \left(\liminf_n \int_{\mathbb{S}^{N-1}}\left(\vertii{\partial^s_{\xi}f \varphi}_{L^p}+ \vertii{\partial^s_{\xi}f_n \varphi}_{L^p}+ Q_n(\xi) \right)^{-N/s} \ d\mathscr{H}^{N-1}(\xi) \right)^{-sp/N}
    \\  &\leq \left( \int_{\mathbb{S}^{N-1}}\liminf _n \left(\vertii{\partial^s_{\xi}f \varphi}_{L^p}+ \vertii{\partial^s_{\xi}f_n \varphi}_{L^p}+ Q_n(\xi) \right)^{-N/s} \ d\mathscr{H}^{N-1}(\xi) \right)^{-sp/N}
    \\ &= \bigg(\int_{\mathbb{S}^{N-1}} \biggl(\vertii{\partial^s_{\xi}f \varphi}_{L^p}  \\ & \hspace{90 pt} + \left(\vertii{\partial^s_{\xi} f \varphi}_{L^p}^p +\int_{\R^N}\verti{\varphi}^p  d \mu_{\xi}   \bigg)^{1/p} \right)^{-N/s} \ d\mathscr{H}^{N-1}(\xi) \bigg)^{-sp/N}.
    \end{split}
    \end{flalign}
    Here, we rely on Fatou's lemma to obtain the second inequality.
    
    Passing to the $\limsup$ in \eqref{sb est} and using \eqref{lims}, we find that
    \begin{flalign} \begin{split}\label{fophi}
    &S^p \left(\int_{\R^N} \verti{\varphi}^q \ d\tilde{\nu} \right)^{p/q} \\ &\leq 
    \biggl(\int_{\mathbb{S}^{N-1}} \bigg(\vertii{\partial^s_{\xi}f \varphi}_{L^p}  \\ & \hspace{90 pt} + \left(\vertii{\partial^s_{\xi} f \varphi}_{L^p}^p +\int_{\R^N}\verti{\varphi}^p  d \mu_{\xi}   \bigg)^{1/p} \right)^{-N/s} \ d\mathscr{H}^{N-1}(\xi) \biggl)^{-sp/N},
    \end{split}
    \end{flalign}
    for each $\varphi \in C_c^{\infty}$ (recall that $\verti{g_n}^q\ dx \overset{*}{\rightharpoonup} \tilde{\nu}$, with $\tilde{\nu}$ defined in \eqref{mes limit}).

    Now consider $\varphi \in C_c^{\infty}(B(0,1))$, such that  $\varphi(0)=1$ and $\vertii{\varphi}_{L^{\infty}}=1$.
    For each $i \in I$, $\varepsilon>0$, we apply \eqref{fophi} to the function $\displaystyle \varphi_{i,\varepsilon}(x)\coloneq \varphi\left(\frac{x-x_i}{\varepsilon}\right)$ and find
    \begin{flalign}\label{epsi esti}
    \begin{split}
     S^p \nu_i^{p/q} &\leq S^p \left(\int_{\R^N} \verti{\varphi_{i,\varepsilon}}^q \ d\tilde{\nu}  \right)^{p/q} 
     \\ & \leq \biggl(\int_{\mathbb{S}^{N-1}} \bigg(\vertii{\partial^s_{\xi}f \varphi_{i,\varepsilon}}_{L^p}  \\ & \hspace{40 pt} + \left(\vertii{\partial^s_{\xi} f \varphi_{i,\varepsilon}}_{L^p}^p +\mu_{\xi}(B(x_i,\varepsilon) )   \bigg)^{1/p} \right)^{-N/s} \ d\mathscr{H}^{N-1}(\xi) \biggl)^{-sp/N}.
    \end{split}
    \end{flalign}
    For each $\xi \in \mathbb{S}^{N-1}$, we have
    \bes
    \mu_{\xi}(B(x_i,\varepsilon)) \to  \mu_{\xi}(\{x_i \}) \ \text{and} \ \vertii{\partial^s_{\xi} f \varphi_{i,\varepsilon}}_{L^p} \to 0, \ \text{as} \ \varepsilon \to 0.
    \ees
    Passing to the $\limsup_{\varepsilon \to 0}$ in \eqref{epsi esti} and using Fatou's lemma, we find
    \begin{flalign}\label{each i}
    \begin{split}
    S^p \nu_{i}^{p/q} &\leq \biggl(\int_{\mathbb{S}^{N-1}} \liminf_{\varepsilon \to 0} \bigg(\vertii{\partial^s_{\xi}f \varphi_{i,\varepsilon}}_{L^p}  \\ & \hspace{60 pt} + \left(\vertii{\partial^s_{\xi} f \varphi_{i,\varepsilon}}_{L^p}^p +\mu_{\xi}(B(x_i,\varepsilon) )   \bigg)^{1/p} \right)^{-N/s} \ d\mathscr{H}^{N-1}(\xi) \biggl)^{-sp/N} 
    \\ & \leq \left(\int_{\mathbb{S}^{N-1}}  \mu_{\xi}(\{x_i\})^{-N/sp} \ d\mathscr{H}^{N-1}(\xi)\right)^{-sp/N},
    \end{split}
    \end{flalign}
    for each $i \in I$.
    Therefore, we find that 
    \bes
    \ba
    \sum_{i \in I} S^{p} \nu_{i}^{p/q} &\leq \sum_{i \in I} \left(\int_{\mathbb{S}^{N-1}}  \mu_{\xi}(\{x_i\})^{-N/sp} \ d\mathscr{H}^{N-1}(\xi)\right)^{-sp/N}
    \\  & \leq \left(\int_{\mathbb{S}^{N-1}}  \left(\sum_{i \in I} \mu_{\xi}(\{x_i\})\right)^{-N/sp} \ d\mathscr{H}^{N-1}(\xi)\right)^{-sp/N},
    \ea
    \ees
    using \eqref{each i} and the reverse Minkowski inequality \eqref{reverse minko}.
\end{proof}
We now recall the first concentration-compactness lemma, due to Lions \cite{lions1984concentration}.
\begin{lemma}\label{1 CC}(\cite[Lemma 1.1]{lions1984concentration}, \cite[4.3] {struwe2000variational})
    Let $(\mu_n)_{n\in \N}$ be a sequence of probability measures. There exists a subsequence $(\mu_{n_k})_{k \in \N}$ such that one of the following holds :
    \begin{enumerate}[(a)]
        \item (Compactness) There exists $(x_k) \subset \R^N$ such that, for each $\varepsilon>0$, there exists $R>0$ satisfying 
        \bes
         \mu_{n_k}(B(x_k,R)) \geq 1-\varepsilon, \ \fo k \in \N.
        \ees
        \item (Vanishing)  For each $R>0$, we have $\sup_{x \in \R^N} \mu_{n_k}(B(x,R)) \to 0$, as $k \to \infty$.
         \item (Dichotomy) There exist $0<\lambda<1$, $(x_k) \subset \R^N$, and $R_k \to \infty$, such that   
        \bes
        \ba
          & \verti{\lambda- \mu_{n_k}(B(x_k,R_k))} + \verti{(1-\lambda)- \mu_{n_k}(\R^N\setminus B(x_k,2R_k))} \to 0, \ \text{as} \ k \to \infty. 
        \ea
        \ees
        \end{enumerate}
\end{lemma}
 The dichotomy case in Lemma \ref{1 CC} is slightly different from the one in \cite[4.3]{struwe2000variational}, but it is a straightforward consequence of this last result, as shown in \cite[p. 47]{struwe2000variational}.

We may now adapt the approach of Lions to the existence of extremal functions in the classical Sobolev inequalities, see \cite[Theorem 8]{lions1985concentration}, or the presentation of this approach in \cite[Theorem 4.9]{struwe2000variational}. 
\begin{proof}[Proof of Theorem \ref{exist sob} in the case where $s$ is an integer]
    Let $(f_n)$ be such that $\vertii{f_n}_{L^q}=1$, for each $n \in \N$, and $\mathscr{E}_{s,p}(f_n) \to S$, as $n \to \infty$. As in the proof of Theorem \ref{exist sob} in the fractional case (see Section \ref{frac}), we may assume (using Lemma \ref{equili}) that
    \be\label{lb 1}
    \verti{f_n}_{W^{s,p}} \leq C \vertii{\partial^s_{\xi}f_n}_{L^p}, \ \fo \xi \in \mathbb{S}^{N-1}, \  \fo n \in \N,
    \ee
    and thus that $(f_n)$ is bounded in $\dot{W}^{s,p}$.
     We may choose  $(x_n)\subset \R^N$ and $(R_n) \subset (0,\infty)$ such that  
    \be\label{kill v}
     \int_{B(0,1)} | \tilde{f_n}(y)|^q \ dy = \sup_{x \in \R^N} \int_{B(x,1)} |\tilde{f_n}(y)|^q  \ dy=1/2,
    \ee
    where $\tilde{f_n}$ is defined in \eqref{def til}. We may replace $f_n$ by $\tilde{f_n}$, as it is done in \cite[I.2. Step 1] {lions1985concentration}, and assume that $(f_n)$ satisfies \eqref{kill v}.
    
    Since $(f_n)$ is bounded in $\dot{W}^{s,p}$, by extracting a subsequence of $(f_n)$, we may further assume that $f_n \rightharpoonup f$ in $\dot{W}^{s,p}$ and that $(f_n)$ satisfies the conclusion of Lemmas \ref{extract} and \ref{affin CC} : 
    \be \label{cl CC}
     \verti{\partial^s_{\xi} f_n}^p \ dx \overset{*}{\rightharpoonup} \verti{\partial^s_{\xi} f}^p + \mu_{\xi}, \ \fo \xi \in \mathbb{S}^{N-1}, \verti{f_n}^q \ dx \overset{*}{\rightharpoonup} \nu=\verti{f}^q \ dx + \sum_{i \in I} \nu_i \delta_{x_i},
     \ee
     \be \label{ineq CC}
    \sum_{i}S^p(\nu_i)^{p/q} \leq \left(\int_{\mathbb{S}^{N-1}} \left( \sum_{i \in I} \mu_{\xi}(\{x_i\})\right)^{-N/sp} \ d\mathscr{H}^{N-1}(\xi) \right)^{-sp/N}.
    \ee
    
    By Lemma \ref{1 CC}, and since we have \eqref{kill v}, we may find a subsequence of $(f_{n})_{n\in\N}$ (still denoted $(f_n)$) such that either item (a) (compactness), or item (c) (dichotomy), holds for the sequence of probability measures $(\verti{f_{n}}^q \ dx)_{n\in \N}$.
    
    Assume first that we have dichotomy. There exist $0<\lambda<1$, $(x_n) \subset \R^N $, and $(R_n) \subset (0,\infty)$ such that
    \be \label{split}
     \int_{B(x_n,R_n)} \verti{f_n}^q \ dx \to \lambda \text{,} \ \int_{\R^N \setminus B(x_n,2R_n)} \verti{f_n}^q \ dx \to 1-\lambda, \ \text{as} \ n \to \infty,
    \ee
    and thus
    \be \label{anneau}
     \int_{B(x_n,2R_n)\setminus B(x_n,R_n)} \verti{f_n}^q \ dx \to 0.
    \ee
    
    We set 
    \be \label{def phi}
    \ba
     \varphi_n(x) \coloneq \varphi\left(\frac{x-x_n}{R_n}\right), \ \fo x\in \R^N,
     \ea
     \ee
     with   $\varphi$  \text{satisfying} $0 \leq \varphi \leq 1, \, \varphi=1$  \text{on} $B(0,1)$  \text{and}  $\supp \varphi \subset B(0,2)$.
    By \cite[(4.14)]{struwe2000variational} and the argument that follows, we have, in that case, 
     \be \label{struw}
     \vertii{D^{i} f_n} \vertii{D^{s-i}\varphi_n} \to 0 \ \text{in} \ L^p, \ \text{as} \ n \to \infty,
     \ee
     for each  $0\leq i <s$. Arguing as in \cite{struwe2000variational}, we may show that 
      \be\label{1t}
    \vertii{\partial_{\xi}^s f_n }_{L^p}^p + \varepsilon_n \geq  \vertii{\partial_{\xi}^s (f_n \varphi_n) }_{L^p}^p + \vertii{\partial_{\xi}^s \left(f_n \left(1-\varphi_n\right)\right) }_{L^p}^p\text{,}
    \ee
     where $\varepsilon_n$ is independent of $\xi \in \mathbb{S}^{N-1}$ and converges to $0$ as $n \to \infty$.  For the convenience of the reader, we reproduce this proof in detail.
    
    For each $\eta>0$, $n \in \N$, and $\xi \in \mathbb{S}^{N-1}$, we have that 
    \be \label{appli of mink}
    \ba
     &\vertii{\partial^{s}_{\xi}(f_n \varphi_n) }_{L^p}^p +  \vertii{\partial^{s}_{\xi}(f_n (1-\varphi_n)) }_{L^p}^p \\ &= \vertii{\partial^s_{\xi}f_n \varphi_n + \sum_{i=0}^{s-1} \binom{s}{i}\partial_{\xi}^{i}f_n \partial_{\xi}^{s-i}\varphi_n }_{L^p}^p +  \vertii{\partial^s_{\xi}f_n (1-\varphi_n) - \sum_{i=0}^{s-1} \binom{s}{i}\partial_{\xi}^{i}f_n \partial_{\xi}^{s-i}\varphi_n }_{L^p}^p
     \\ & \leq (1+\eta)\left(\vertii{\partial^{s}_{\xi} f_n \varphi_n}_{L^p}^p + \vertii{\partial^{s}_{\xi} f_n (1-\varphi_n)}_{L^p}^p \right) +C(\eta) \vertii{ \sum_{i=0}^{s-1} \binom{s}{i}\partial_{\xi}^{i}f_n \partial_{\xi}^{s-i}\varphi_n}_{L^p}^p,
    \ea
    \ee
    using \eqref{mink eps}. Since $0 \leq \varphi_n \leq 1$, we also have that 
    \bes
     \vertii{\partial^{s}_{\xi} f_n \varphi_n}_{L^p}^p + \vertii{\partial^{s}_{\xi} f_n (1-\varphi_n)}_{L^p}^p \leq \vertii{\partial^s_{\xi} f_n}_{L^p}^p.
    \ees
    Combining the last inequality with \eqref{appli of mink}, we find that, for each $\eta>0$, $n \in \N$, and $\xi \in \mathbb{S}^{N-1}$,
    \bes
    \ba
    \vertii{\partial^{s}_{\xi}(f_n \varphi_n) }_{L^p}^p &+  \vertii{\partial^{s}_{\xi}(f_n (1-\varphi_n)) }_{L^p}^p \\ &\leq (1+\eta) \vertii{\partial^s_{\xi} f_n}_{L^p}^p +  C(\eta) \vertii{ \sum_{i=0}^{s-1} \binom{s}{i}\partial_{\xi}^{i}f_n \partial_{\xi}^{s-i}\varphi_n}_{L^p}^p 
    \\ & \leq (1+\eta) \vertii{\partial^s_{\xi} f_n}_{L^p}^p+ C(\eta) \sum_{i=0}^{s-1}\vertii{\vertii{D^{i} f_n} \vertii{D^{s-i}\varphi_n}}_{L^p}^p.
    \ea
    \ees
    Thus, for each $\varepsilon>0$, we may fix $\eta>0$ such that $\eta \sup_{n \in \N} \verti{f_n}_{W^{s,p}} < \varepsilon$ (since $(f_n)$ is bounded in $\dot{W}^{s,p}$), and find $N \in \N$ (independent of $\xi \in \mathbb{S}^{N-1}$) such that 
    \bes
     \vertii{\partial^{s}_{\xi}(f_n \varphi_n) }_{L^p}^p +  \vertii{\partial^{s}_{\xi}(f_n (1-\varphi_n))}_{L^p}^p \leq  \vertii{\partial^s_{\xi} f_n}_{L^p}^p +2\varepsilon 
    \ees
    for each $n \geq N$, using \eqref{struw}. This proves \eqref{1t}.

    Since $\vertii{f_n}_{L^q}=1$,  $\vertii{\partial^s_{\xi}f_n}_{L^p}\geq C>0$, for each $n$ and $\xi \in \mathbb{S}^{N-1}$,  by \eqref{lb 1} and the Sobolev inequality.  Thanks to this fact, we may apply the mean value inequality to the functions $t \mapsto x^{-N/sp}$ and $x \mapsto t^{-sp/N}$ successively, and find
    \be\label{2t}
      \left( \int_{\mathbb{S}^{N-1}} \left( \vertii{\partial_{\xi}^s f_n }_{L^p}^p + \varepsilon_n \right)^{-N/sp} \ d\mathscr{H}^{N-1}(\xi)\right)^{-sp/N} \leq \mathscr{E}_{s,p}(f_n)^p + o_{n\to \infty}(1). 
    \ee 
     Using the reverse Minkowski inequality \eqref{reverse minko}, and combining \eqref{1t} and \eqref{2t}, we find
    \be\label{penul}
    \ba
    &\mathscr{E}_{s,p}(f_n\varphi_n)^p + \mathscr{E}_{s,p}(f_n(1-\varphi_n))^p 
    \\ & \leq  \left( \int_{\mathbb{S}^{N-1}} \left(\vertii{\partial_{\xi}^s (f_n \varphi_n) }_{L^p}^p + \vertii{\partial_{\xi}^s \left(f_n \left(1-\varphi_n\right)\right) }_{L^p}^p \right)^{-N/sp} \ d\mathscr{H}^{N-1}(\xi)\right)^{-sp/N}
    \\ & \leq \left( \int_{\mathbb{S}^{N-1}} \left( \vertii{\partial_{\xi}^s f_n }_{L^p}^p + \varepsilon_n \right)^{-N/sp} \ d\mathscr{H}^{N-1}(\xi)\right)^{-sp/N} 
    \\ & \leq \mathscr{E}_{s,p}(f_n)^p + o_{n\to \infty}(1). 
    \ea
    \ee
Using \eqref{penul}, the affine Sobolev inequality \eqref{aff sob}, and the definition of $\varphi_n$ \eqref{def phi}, we thus find
\be \label{last dichotomy}
\ba
&\mathscr{E}_{s,p}(f_n)^p+o_{n\to \infty}(1) \geq S^p \left(\vertii{f_n \varphi_n}_{L^q}^p +   \vertii{f_n (1-\varphi_n)}_{L^q}^p \right)
\\ & \geq S^p\left(\left(\int_{B(x_n,R_n)} \verti{f_n}^q dx\right)^{p/q} +\left(\int_{\R^N \setminus B(x_n,2R_n)} \verti{f_n}^q dx\right)^{p/q} \right).
\ea
\ee
 Passing to the limit in \eqref{last dichotomy}, using \eqref{split}, we find
\bes
S^p \geq S^p \left( \lambda^{p/q}+(1-\lambda)^{p/q}\right)>S^p,
\ees
since $p<q$ and $0<\lambda<1$. This yields a contradiction, and shows that dichotomy cannot occur.

Hence, by Lemma \ref{1 CC}, there is compactness for the sequence of probability measures $\displaystyle (\verti{f_n}^q \ dx)_{n\in \N}$: there exists $(x_n) \subset \R^N$ such that, for each $\varepsilon>0$, there exists $R(\varepsilon)$ satisfying
\be\label{compact}
\int_{B(x_n,R(\varepsilon))} \verti{f_n}^q \ dx \geq 1-\varepsilon, \ \fo n \in \N.
\ee
For each $0<\varepsilon<1/2$, \eqref{kill v}, \eqref{compact}, and the fact that $\vertii{f_n}_{L^q}=1$, imply that $B(0,1)\cap B(x_n, R(\varepsilon))\neq \emptyset$, for each $n \in \N$. Hence, we also have
$B(x_n,R(\varepsilon)) \subset B(0,2R(\varepsilon)+1)$, and
\bes
\int_{B(0,2R(\varepsilon)+1)} \verti{f_n}^q \ dx \geq 1-\varepsilon, \ \fo n \in \N.
\ees
  We obtain that $(\verti{f_n}^q \ dx)$ is a tight sequence of probability measures, converging to a measure $\nu$, and thus 
  \be\label{proba}
  \nu \ \text{is a probability measure}. 
  \ee

By \eqref{lb 1}, we have $\vertii{\partial_{\xi}^s f_n}_{L^p} \geq C>0$, for each $\xi \in \mathbb{S}^{N-1}$, $n \in \N$.  Hence, Lemma \ref{rev fatou} yields
\be \label{limsup}
\ba
\limsup_n \int_{\mathbb{S}^{N-1}}\vertii{\partial_{\xi}^s f_n}_{L^p}^{-N/s} \ d\mathscr{H}^{N-1}(\xi) &\leq \int_{\mathbb{S}^{N-1}} \limsup_n \vertii{\partial_{\xi}^s f_n}_{L^p}^{-N/s} \ d\mathscr{H}^{N-1}(\xi)
\\&=\int_{\mathbb{S}^{N-1}}  \left(\liminf_{n} \vertii{\partial_{\xi}^s f_n}^p_{L^p}\right)^{-N/sp} \ d\mathscr{H}^{N-1}(\xi).
\ea
\ee
On the other hand, for each $\xi \in \mathbb{S}^{N-1}$,
\be \label{fatou weak}
\vertii{\partial^s_{\xi} f}_{L^p}^p+ \sum_{i \in I} \mu_{\xi}(\{x_i \}) \leq \vertii{\partial^s_{\xi} f}_{L^p}^p+ \mu_{\xi}(\R^N) \leq  \liminf_n \vertii{\partial^s_{\xi} f_n}_{L^p}^p
\ee
by \eqref{cl CC}. 
Combining \eqref{limsup} and \eqref{fatou weak}, we have
\bes
\limsup_n \int_{\mathbb{S}^{N-1}}\vertii{\partial_{\xi}^s f_n}_{L^p}^{-N/s} \ d\mathscr{H}^{N-1}(\xi) \leq \int_{\mathbb{S}^{N-1}}  \left(\vertii{\partial^s_{\xi} f}_{L^p}^p+ \sum_{i \in I} \mu_{\xi}(\{x_i \})\right)^{-N/sp} \ d\mathscr{H}^{N-1}(\xi), 
\ees
and, raising the last inequality to the $-sp/N$ power, we find
\bes
\liminf_n\mathscr{E}_{s,p}(f_n)^p=S^p \geq \left(\int_{\mathbb{S}^{N-1}}  \left(\vertii{\partial^s_{\xi} f}_{L^p}^p+ \sum_{i \in I} \mu_{\xi}(\{x_i \})\right)^{-N/sp} \ d\mathscr{H}^{N-1}(\xi) \right)^{-sp/N}.
\ees
Therefore, by the reverse Minkowski's inequality \eqref{reverse minko} and \eqref{ineq CC}, we have
\be\label{l com}
S^p \geq \mathscr{E}_{s,p}(f)^p + \sum_{i \in I}S^{p} \nu_i^{p/q}\geq S^{p}\left(\left(\int_{\R^N} \verti{f}^q \ dx\right)^{p/q} + \sum_{i \in I} \nu_i^{p/q}\right).
\ee
Since $\nu$ is a probability measure \eqref{proba}, $\int_{\R^N} \verti{f}^q \ dx + \sum_{i \in I} \nu_i=1$. Note that condition \eqref{kill v} implies that $\nu_i <1$, for each $i \in I$. Hence, if there exists $i \in I$ such that $\nu_i>0$, we find that
\bes
\left(\left(\int_{\R^N} \verti{f}^q \ dx \right)^{p/q}+ \sum_{i \in I} \nu_i^{p/q}\right) > \bigg( \int_{\R^N} \verti{f}^q \ dx + \sum_{i \in I} \nu_i\bigg)^{p/q}=1.
\ees
since $p/q<1$. By \eqref{l com}, this yields a contradiction. 

We may conclude that $\nu_i=0$, for each $i \in I$, and that $\vertii{f}_{L^q}=1$. 
As in the fractional case, by the affine Sobolev inequality \eqref{aff sob} and Lemma \ref{weak lsc}, we have
\bes
S \leq \mathscr{E}_{s,p}(f) \leq \liminf_n \mathscr{E}_{s,p}(f_n) = S,
\ees
  hence $f$ is an optimizer for the affine Sobolev inequality. 
 
 Moreover, since $\mathscr{E}_{s,p}(f_n)\to \mathscr{E}_{s,p}(f)$ and $f_n \rightharpoonup f$ in $\dot{W}^{s,p}$, we may conclude that $f_n \to f$ in $\dot{W}^{s,p}$, by Lemma \ref{not in 0}.
\end{proof}

\appendix
\section{Proof of Lemma \ref{non zero ex} and application to the existence of extremal functions in higher-order fractional Sobolev inequalities}\label{A}
In \cite{zhang2021optimizers}, Zhang studies the existence of extremal functions in the fractional Sobolev inequality when $0<s<1$ and $1<p<\infty$. His proof yields the following result. 
\begin{theo}\label{classic ex}
     Let $s>0$ be non-integer and $1 \leq p<\infty$ be such that $sp<N$. There exists $f \in \mathring{W}^{s,p}$ such that $\vertii{f}_{L^q}=1$ and $\verti{f}_{W^{s,p}}= \inf\{ \verti{g}_{W^{s,p}}; \, \vertii{g}_{L^q}=1 \}$.
\end{theo}
We present this proof for the sake of completeness, and start with the one of Lemma \ref{non zero ex}, relying on the argument in \cite[Lemma 3.2]{zhang2021optimizers}.
\begin{proof}[Proof of Lemma \ref{non zero ex}]
    Since $(f_n)$ is bounded in $\mathring{W}^{s,p}$ and $\inf_n \vertii{f_n}_{L^q}>0$, using \eqref{improved}, we find that
    \be\label{inf bd}
     \alpha \coloneq \inf_n \sup_{A>0}A^{-N/q} \vertii{\mathscr{F}^{-1}\left(\Psi(A^{-1} \cdot) \mathscr{F}f_n\right)}_{L^{\infty}} >0,
    \ee
using the notation introduced in \eqref{improved}. Set $\varphi \coloneq \mathscr{F}^{-1}\Psi$ (note that $\varphi \in \mathscr{S}$). By \eqref{inf bd}, for each $n \in \N$, there exist $A_n>0$ and $x_n \in \R^N$, such that
\be\label{lbnd}
A_n^{N-N/q}\verti{(\varphi(A_n\cdot) \ast f_n) (x_n)} \geq \alpha/2. 
\ee
But 
\be\label{formula}
(A_n^{N-N/q}\varphi(A_n\cdot) \ast f_n) (x_n)= \int_{\R^N} \varphi(y) A_n^{-N/q}f_n(x_n-y/A_n) \ dy= \int_{\R^N} \varphi(y) \tilde{f_n}(y) \ dy, 
\ee
where $\tilde{f_n}(x) \coloneq A_n^{-N/q} f_n(x_n-x/A_n),$ for each $x\in \R^N$.

We have $\verti{\tilde{f_n}}_{W^{s,p}}=\verti{f_n}_{W^{s,p}}$, for each $n\in \N$, therefore $(\tilde{f_n})$ is bounded in $\mathring{W}^{s,p}$ and there exists $f \in \mathring{W}^{s,p}$ such that $f_n \rightharpoonup f \in \mathring{W}^{s,p}$ and in $L^q$, up to a subsequence, by Lemma \ref{weak comp}. We still denote by $(f_n)$ this subsequence, and we have
\be\label{limit}
\int_{\R^N} \varphi(y) f_n(y) \ dy \to \int_{\R^N} \varphi(y) f(y) \ dy, 
\ee
since $\varphi \in \mathscr{S}$. Hence, combining \eqref{lbnd}, \eqref{formula} and \eqref{limit}, we find that 
\bes
\verti{\int_{\R^N} \varphi(y) f(y) \ dy} \geq \alpha/2,
\ees
which implies that $f\neq 0$.
\end{proof}
We now turn to the 
\begin{proof}[Proof of Theorem \ref{classic ex}]
    Let $(f_n) \subset \mathring{W}^{s,p}$ be such that $\vertii{f_n}_{L^q}=1$ and 
    \bes \verti{f_n}_{W^{s,p}} \to I\coloneq \inf \{ \verti{g}_{W^{s,p}}; \, \vertii{g}_{L^q}=1 \}
, \ \text{as} \ n \to \infty.
    \ees
    Starting from $(f_n)$, we may define a sequence $(\tilde{f_n})$ such that $\|\tilde{f_n}\|_{L^q}=1$, $\|\tilde{f_n}\|_{W^{s,p}} \to I$ and $\tilde{f_n} \rightharpoonup f \neq 0$ in $\mathring{W}^{s,p}$, as $n \to \infty$ (as in \eqref{def til}). We still denote it by $(f_n)$.  
    
    By further extraction, we may also assume that $f_n \to f$ a.e., and find
    \be\label{BL1}
    \verti{f_n}_{W^{s,p}}^p= \verti{f_n-f}_{W^{s,p}}^p+\verti{f}^p_{W^{s,p}} + o(1),
    \ee
by an application of the Brezis-Lieb lemma to the sequence
\bes
\R^N \times \R^N \ni (h,x) \mapsto \frac{\Delta_{h}^{\lfloor s \rfloor +1} f_n(x)}{\verti{h}^{s+N/p}}.
\ees
Another application of the Brezis-Lieb lemma shows that
\bes\label{BL2}
\vertii{f_n}_{L^q}^q= \vertii{f_n-f}_{L^q}^q +\vertii{f}_{L^q}^q + o(1)
\ees
and thus that 
\be\label{lim fact}
\lim_n \vertii{f_n-f}_{L^q}^q +\vertii{f}_{L^q}^q=1.
\ee
By \eqref{BL1} and the Sobolev inequality, we have
\bes
\ba
\verti{f_n}_{W^{s,p}}^p &\geq I^p(\vertii{f_n-f}_{L^q}^p + \vertii{f}_{L^q}^p) +o(1).
\ea
\ees
Hence, if $\limsup_n \vertii{f_n -f}_{L^q} >0$, we find
\bes
\ba
I^p = \lim_n \verti{f_n}_{W^{s,p}}^p &\geq  I^p\left(\limsup_n \vertii{f_n-f}_{L^q}^p + \vertii{f}_{L^q}^p\right)
\\ &> I^p\left(\limsup_n \vertii{f_n-f}_{L^q}^q + \vertii{f}_{L^q}^q\right)^{p/q}
\\ & = I^p,
\ea
\ees
and a contradiction. Here, we rely on \eqref{n fact}, $\vertii{f}_{L^q}>0$, and $\limsup_n \vertii{f_n -f}_{L^q} >0$ for the second inequality, while we use \eqref{lim fact} to obtain the equality.  

We may now conclude, arguing as in the fractional case of the proof of Theorem \ref{exist sob}. \qedhere
\end{proof}

\bibliography{ext} 

\end{document}